\documentclass{amsart}

\usepackage{slashed, mathrsfs,  dutchcal, mathabx, esint}

\newcommand{\mc}{\mathcal}

\newcommand{\mb}{\mathbb}

\newcommand{\mr}{\mathring}
\newcommand{\wh}{\widehat}

\newcommand{\wt}{\widetilde}

\newcommand{\ol}{\overline}

\title{Estimates for the Constant Mean Curvature Dirichlet Problem on Catenoidal necks}

 \usepackage{amsmath, amssymb, amsfonts, amsthm, amscd, graphicx, subfiles, xr, leftidx, slashed, enumitem, lipsum } \usepackage{mathrsfs}
 \usepackage{leftidx}
  \usepackage{xcolor}
  \usepackage{graphicx}

  \usepackage{verbatim}

\long\def\symbolfootnote[#1]#2{\begingroup%
\def\thefootnote{\fnsymbol{footnote}}\footnote[#1]{#2}\endgroup}

\theoremstyle{plain}
\numberwithin{equation}{section}

\newtheorem{theorem}{Theorem}[section]  
\newtheorem{lemma}[theorem]{Lemma}
\newtheorem{corollary}[theorem]{Corollary}
\newtheorem{proposition}[theorem]{Proposition}

\newtheorem{remark}[theorem]{Remark}

\title{Estimates for the Constant Mean Curvature Dirichlet Problem on catenoids}
\author{Stephen J.  Kleene}
\address{Department of Mathematics, University of Rochester, Rochester, NY}
\email{skleene@ur.rochester.edu}

\begin{document}
\maketitle

\begin{abstract}   In this article, we solve the constant mean curvature dirichlet problem on catenoidal necks with small scale  in  $\mb{R}^3$. The solutions are found in exponentially weighted H\"older spaces with non-integer weight and  are a-priori  bounded by a uniform constant times $r^{1 + \gamma}$, where $r$ denotes the distance to the axis of the neck and where $\gamma$ belongs to the interval $(0, 1)$. By comparing the solutions with their limits on the disk, we improve the estimate to $\gamma  =1$. As a corollary, we prove differentiability of solutions in $\tau$ down to $\tau = 0$. The surfaces we construct have applications to gluing constructions. 
\end{abstract}
\section{Introduction}

  A \emph{catenoidal neck} in $\mb{R}^3$ is the intersection of a catenoid with a tube about its axis of rotation. Up to scalings and rigid motions, catenoidal necks comprise a one parameter family $\mc{N}_\tau = \tau \mc{C} \cap \{r \leq 1 \}$,
where here $\mc{C}$ denotes the \emph{standard catenoid} implicitly given by the equation $r = \cosh(z)$ in $\mb{R}^3$ and where $r = \sqrt{x^2 + y^2}$ denotes the distance to the $z$-axis. We refer to the parameter $\tau$ as the \emph{scale} of the catenoidal neck $\mc{N}_\tau$. Let $\mc{N}_{\tau}^\pm$ denote the intersection of $\mc{N}_\tau$ with the half space $\{\pm z\geq 0 \}$.  Then $\mc{N}^\pm_{\tau}$ converges in $C^{k}$  to the disk $\mc{D}$ of radius $1$  in the plane $\{ z = 0\}$  away from the origin as $\tau \rightarrow 0$.  The sheets  $\mc{N}^\pm_{\tau}$ are smooth modulo vertical translations in the sense that the translation
  \[
  \wt{\mc{N}}^+_{\tau} = \mc{N}_\tau^+ - \tau\text{arccosh}(1/\tau) e_z
  \]
  of $\mc{N}_{\tau}$ converges smoothly to the disk $\mc{D}$  away from the origin in $C^k$, and the analogous statement holds for $\mc{N}^-$. In order to avoid introducing additional notation, we identify $\wt{\mc{N}}^\pm = \mc{N}^\pm$ throughout this article. The variation fields generated by $\mc{N}^{\pm}_{\tau}$ at $\tau$ = 0 are respectively $- \log(r)$ and $\log(r)$, where $r$ is the polar distance on $\mb{R}^2$ 

In this article, we study the constant mean curvature Dirichlet problem  on catenoidal necks of small scale. For $\tau$ sufficiently small, the surfaces $\mc{N}_{\tau}$ are non-degenerate, and thus the CMC dirichlet problem is solvable for small boundary data and small mean curvature.  Since the family $\mc{N}_{\tau}$ is divergent in the moduli space, there is a-priori no uniform estimate on the smallness required for solvability across the family, and in fact there are obstructions arising from  lower modes--functions in the span of $1$, $\cos(x)$ and $\sin(x)$--on the circle.  We show that the Dirichlet problem is uniformly solvable up to lower modes across the family as the scale tends to zero. Since the boundary values of our solutions are not completely determined, in order to uniquely characterise our solutions we demand that the solutions satisfy a set of orthogonality conditions along the waist of the necks, and a function satisfying these conditions is said to be \emph{normalized}.  Our main theorem can them be stated as follows:

\begin{theorem}\label{Prop:MainTheoremSimple}
There is $\epsilon> 0$ and $C > 0$ such that: Given a function $f \in C^{2, \alpha} (\partial \mc{N}_\tau)$  and $\delta \in (-\epsilon, \epsilon )$, there is a unique normalized  function $\mc{h}_{\tau, \delta, f}$ in $C^{2, \alpha} (\mc{N}_\tau)$ satisfying the estimate
\[
|\mc{h}_{\tau, \delta, f}| \leq^C r^2\| f\|
\] 
and such that the normal graph $\mc{N}_{\tau, \delta, f}$ over $\mc{N}_\tau$ has constant mean curvature $\delta$ and such that the trace of $\mc{h}_{\tau, \delta, f}$ agrees with $f$ up to lower modes. 
 \end{theorem}
 
 Our solutions are found by solving the Poisson-Dirichlet problem for the Jacobi operator of the catenoid, or rather its pullback to the cylinder under its conformal parametrization, in exponentially weighted H\"older spaces. Similar problems arise in various formulations in numerous articles (\cite{Kapouleas}, \cite{Kapouleas2},  \cite{Kapouleas3}, \cite{Kapouleas-Mcgrath}, \cite{mazzeo-pacard}, \cite{mazzeo-pacard2}, \cite{mazzeo-pacard-pollack}).  A standard feature of such problems is that rates of exponential decay must be a non-integer. Our solutions are found in such spaces and thus a-priori satisfy the uniform estimate $|\mc{h}_{\tau, \delta, f}| \leq^C r^{1 + \gamma}\| f\|$.
 
 The refinement offered by Theorem \ref{Prop:MainTheoremSimple}  lies in the improvement of the decay estimate for the solutions to $\gamma = 1$,  which allows us to prove  differentiability of the solutions in $\tau$ at $\tau = 0$. We obtain these improved estimates by comparing our solutions with their limits  on this disk. 

The family of solutions $\mc{h}_{\tau, \delta, f}$ extends differentiably in $\tau$ to $\tau = 0$ in the following sense:
If $\mc{h}^\pm_{\tau, \delta, f}$ denotes the restriction of $\mc{h}_{\tau, \delta, f}$ to $\mc{N}^{\pm}_{\tau}$, then $\mc{h}^\pm_{\tau, \delta, f}$ converges in $C^{2, \alpha/2}$ away from the origin to a constant mean curvature graph  over the disk satisfying the estimate
\[
\| \mc{h}^{\pm}_{\delta, f}\| \leq^C r^2\|f^\pm \|
\]
and with trace equal to $f^+$--the restriction of $f$ to the top boundary component  of $\mc{N}^{+}_\tau$--up to lower modes. Since such solutions are unique, the family $\mc{h}_{\tau, \delta, f}$ extends continuously in  $C^{2, \alpha}$ to $\tau = 0$ in the sense that the family of  functions $ \mc{h}^{\pm}_{\tau, \delta, f}$ extend continuously in $C^{2, \alpha}$ away from the origin to $\tau = 0$.  In order to show that the family extends differentiably in $\tau$ to $\tau = 0$, we fix a family of diffeomorphisms $\phi_{\tau_0}^{\tau}: \mc{N}_{\tau_0} \rightarrow \mc{N}_{\tau}$ for $\tau$ near $\tau_0$, which allows us to differentiate the solutions in $\mc{h}_{\tau, \delta, f}$ in $\tau$ by fixing a common domain for $\tau$ near $\tau_0$. The mappings $\phi_{\tau_0}^{\tau}$ are obtained by representing $\mc{N}_\tau$ as a normal graph over $\mc{N}_{\tau_0}$ after slightly adjusting the scale of $\mc{N}_{\tau}$. The maps 
 converge to a limiting map as $\tau_0 \rightarrow 0$ in the following sense: Let $\phi_{\tau_0}^{\tau, \pm}$ denote the restriction of $\phi_{\tau_0}^{\tau}$ to the graphical sheet $\mc{N}^{\pm}_{\tau_0}$. Then $\mc{N}^{\pm}_{\tau_0}$ converges in $C^{k}$ on away from the origin to the disk $\mc{D}$,  and the maps $\phi_{\tau_0}^{\tau, \pm}$ converge away from the origin in $C^{k}$ to the parametrization of $\mc{N}^{\pm}_{\tau}$ as a graph over $\mc{D}$. 
 The functions $\mc{h}_{\tau, \delta, f}$ then depend smoothly on $\tau$ for $\tau >0$ and   we define the $\tau$-derivative of $\mc{h}_{\tau, \delta ,f}$ at $\tau_0$ to be
 \[
 \dot{\mc{h}}_{\tau_0, \delta, f} : = \left. \frac{d}{d \tau} \right|_{\tau = \tau_0} \left( \mc{h}^{\tau}_{\tau_0, \delta, f}\right), \quad \mc{h}_{\tau_0, \delta, f}^{\tau} : = \mc{h}_{\tau, \delta, f} \circ \phi_{\tau_0}^{\tau} 
 \]
 The Gateaux derivative in $\tau$ at $\tau = 0$ is given by
 \[
 \dot{\mc{h}}^{\pm}_{\delta, f}  : =   \lim_{\tau \rightarrow 0} \frac{\mc{h}^{\pm}_{\tau, \delta, f} - \mc{h}^{\pm}_{ \delta, f}}{\tau}
 \] 
and is well defined as a limit in $C^{2, \alpha/2}$ on  compact subsets  $\mc{D}$ of the unit disk way from the origin. We then show:
\begin{theorem} \label{Prop:MainTheoremDifferentiability}
The functions $\dot{\mc{h}}^{\pm}_{\tau_0, \delta, f}$ are uniformly bounded in $C^{2, \alpha}$ away from the origin and converge  as $\tau_0 \rightarrow 0$ to $\dot{\mc{h}}^{\pm}_{\delta, f}$.
\end{theorem}
As an immediate corollary,  the surfaces $\mc{N}_{\tau, \delta, f}$ depend differentiably on $\tau$ at $\tau = 0$, in the following sense: Letting $\mc{N}^{\pm}_{\tau, \delta, f}$ denote the normal graph of $\mc{h}^{\pm}_{\tau, \delta, f}$ over $\mc{N}_{\tau}^{\pm}$, the  surface $\mc{N}^{\pm}_{\tau, \delta, f}$ depends smoothly on $\delta$ and $f$, smoothly in $\tau$ for $\tau > 0$ and once differentiably in $\tau$ at $\tau = 0$, in $C^{2, \alpha}$ on compact subsets away from the origin. In particular, the family  has a $C^1$ extension to $\tau = 0$.  Continuous extensions  for the derivatives in $\delta$ and $f$  are easy and most of the technical work of the paper is devoted to showing that the derivative in $\tau$ extends continuously.

\subsection{General comments about notation, conventions, and the proof}
\subsubsection{Notation, terminology and other convections}
We do not define within this article widely used notation and terminology  that is considered standard, such as, for example, H\"older spaces $C^{k, \alpha} (D)$ or standard objects associated with euclidean spaces. We use the standard practice that $C$  and $\epsilon$ denote any number of constants that must be chosen large or small subject to various consideration. Thus, for example the reader should regard an `$\epsilon$', when it is encountered, as the minimum of the previous encountered `$\epsilon$' along with any additional constraints required. Since we are not interested in quantifying  the dependence of our solutions on any of these constants, we do not distinguish them. We use the non-standard notation $f \leq^C g$ to mean $f \leq C g$, which is convenient in the case that the right hand side contains many terms and would otherwise need to be enclosed in parentheses. 

Frequently, the surfaces we consider in this article come with a specified parametrization. When this is case, we do not carefully distinguish between objects defined on the surface and the  pullbacks under the parametrization. Whether we are considering a function as defined on a surface or on the parametrization domain will be clear from context. 

\subsubsection{General comments of the proof}
 The central objects we consider are geometric operators such as the mean curvature and unit normal operators $F \mapsto H_F$ and $F \mapsto N_F$ defined on immersions $F$, and the resulting remainder terms that arise when perturbing immersions  by graphs. Rather than write these terms explicitly, we appeal abstractly to the fact that these operators are smooth--in the sense of Frechet differentiability--as mappings on  spaces of class $C^{2, \alpha}$ H\"older immersions, and are thus locally bounded in $C^k$  in their domains. Standard integral formulae then apply for taylor remainders and most of the required estimates follow easily.  A few of the more delicate estimates require deeper analysis, and we pause here to observe them so they are not a surprise to the reader later. Firstly in places--such as the proof of Lemma \ref{E1Estimates}--we will need the fact that the hessian of the graph mean curvature operator  vanishes  at affine maps. This allows for improved decay estimates for various terms since the standard conformal parametrization of the catenoid is exponentially flat. 
\subsubsection{Arguments requiring mode preservation} \label{Sec:ModePreservation}
The most delicate estimates in the paper arise in the proof of Propositions \ref{Prop:ImprovedDecay} and  \ref{E1Estimates}. They require a separation of the error term into the lower and higher parts, with the higher part satisfying improved bounds relative to the whole terms. In both these cases, the rotational symmetry of the catenoid is exploited in the following way: Let $F$ be an immersion of a domain $D = \mb{S}^1 \times I$ with coordinates $(x, s)$ into $\mb{R}^3$ and $u$ a function on $D$.  Then the mean curvature $H_F(u)$ of the normal graph over $F$ by $u$   can be computed as:
\[
H_{F} (u)  = \Phi (\mc{J}_u(x, s), x, s)
\]
where $\mc{J}_u = \left(u, \nabla u, \nabla^2 u \right)$ is the vector of partial derivatives of $u$ of  up to order $2$, and where $\Phi$ is a smooth function defined on an open subset of a euclidean space about the origin. Suppose now that $F$ is an immersion of rotation, so in cylindrical coordinates, we have:
\[
F(x, s) = r(s) e_r(x) + h(z) e_z, 
\]
where above we have set $e_r(x) = \cos(x) e_x + \sin(x) e_x$.  Then the symmetry of the parametrization $F$ implies that the function $\Phi$ does not depend on $x$, so
\[
 \Phi (\mc{J}, x, s) = \Phi(\mc{J}, s).
\]
We then have
\begin{align*}
\int_{\mb{S}^1} e^{- ik x}\left. D H_F \right|_{0} (u) & = \int_{\mb{S}^1} e^{i k x} \left. D \Phi \right|_{\left(0, s \right)}\left(\mc{J}_{u}, s \right) \\
&\left. D \Phi \right|_{\left(0, s \right)}\left( \int_{\mb{S}^1} e^{i k x} \mc{J}_{u}, s \right) 
\end{align*}
This allows for bounds to be estimated for modes separately.  This observation generalizes the following way, which is needed in the proof of Proposition \ref{E1Estimates}. Assume now that $\xi = \xi(s)$ is a rotationally symmetric function on $D$. Then,  $\mc{J}_{\xi}$ is a function only of $s$ and we have
\begin{align*}
\int_{-\pi}^\pi e^{i k x}\left. D^{2} H \right|_{u} \left( \xi, u\right)  dx & = \int_{-\pi}^\pi e^{i k x} \left. D^{2} \Phi \right|_{(0, s)}\left( \mc{J}_{\xi},e^{i k x}\mc{J}_{u} \right)  dx \\
& = \left. D^{2} \Phi \right|_{(0, s)}\left( \mc{J}_{\xi},  \int_{-\pi}^\pi e^{i k x}\mc{J}_{u} dx\right)
\end{align*}
We refer this property of the operators $ \left. D H_F\right|_{0} (0)$ and $\left. D^2 H_F \right|_{0} (\xi, -)$ as \emph{mode preserving}. The mode preserving property is critical for our required estimates.

\subsubsection{Other comments}
In order to keep this article as self-contained as possible, we have included here the proofs of several results that are by this point fairly standard in singular perturbation literature. In particular, our treatment of the mapping properties of the jacobi operator on the catenoid in exponentially weighted H\"older spaces closely follows similar proofs  found in \cite{mazzeo-pacard}. The only external machinery we rely upon is what is frequently referred to as `standard theory', specifically, the existence and regularity theory for second order elliptic operators on bounded domains (c.f. \cite{Gilbarg-Trudinger}), the Implicit Function Theorem, and some basic properties of differentiability in Banach spaces. Specifically, our argument relies on the fact a mapping which is continuously Gateaux differentiable on a domain is Frechet differentiable (see, for example \cite{andrews-hopper} Proposition A.3). 

Finally, we remark that we are considering Dirichlet-Poission  problems  on the catenoid in which no symmetry is presumed. In presence of enough symmetries--for example, if we require invariance under reflections through the $x$ and $y$ axis, then the  challenging estimates  required in  our proof become trivial and go through automatically.  This is because such symmetry requirements kill the problematic lower modes, and the remaining ones necessarily satisfy improved decay  rates.

 \section{Preliminaries and Background}

\subsection{Geometric objects as smooth mappings} \label{Sec:GeometricObjects}
The basic objects we consider in this article are  spaces of class $C^{k, \alpha}$ immersions into $\mb{R}^3$ as well as geometric quantities associated to oriented  immersions such as the unit normal and mean curvature. These geometric quantities can be regarded as smooth mappings in the sense of Frechet differentiability on space of class $C^{k, \alpha}$ immersions as follows: Given a smooth compact domain $D$ in $\mb{R}^2$, let $R^{k, \alpha} (D) \subset C^{k, \alpha} (D, \mb{R}^3)$ denote the space of class $C^{k, \alpha}$ immersions of $D$ into $\mb{R}^3$. We require immersions to extend as immersions to an open set containing $D$, so that $R^{k,\alpha} (D)$ is an open subset of $C^{k, \alpha}(D, \mb{R}^3)$. The  mean curvature mapping $H: F \mapsto H_F$, taking an immersion $F$ to its mean curvature $H_F$ is then smooth as a map from the space $R^{k +2, \alpha} (D, \mb{R}^3)$ into $C^{k, \alpha} (D)$. Similarly, the unit normal mapping $N: F \mapsto N_F$ taking an immersion to its unit normal is smooth as a mapping from $R^{k+ 1, \alpha} (D)$ into $C^{k, \alpha} (D, \mb{R}^3)$. We have omitted the proof of these facts, but they are easy to establish by, for example, computing the mean curvature as a smooth function of components of the derivative of an immersion, so
\[
H_F = \ol{H}(\nabla F, \nabla^2 F)
\]
where $\ol{H} = \ol{H} (q_{i}, r_{i, j})$  is defined and smooth on the open subset  of the euclidean space $\left(\mb{R}^3 \right)^{\times 2} \times \left(\mb{R}^3 \right)^{\times 4}$ with  $\det(q_i \cdot q_j ) \neq 0$.
 Smoothness in this sense then works as it does in the standard euclidean setting. In particular, the usual integral expressions for Taylor remainders holds on locally convex domains in $R^{k, \alpha} (D)$.
 
 We let $H_F$ and $N_F$ denote the mean curvature of a $C^{2}$ immersion of $F: D \rightarrow \mb{R}^3$, which we regard as functions on $D$. Given a function $u: D \rightarrow \mb{R}$, the normal graph over $F$ by $u$ is the mapping
\[
F_u : = F + u N_F
\]
Provided $F_u$ is a $C^{2}$ immersion again, we let $H_F(u)$ denote its mean curvature. When $D$ is compact, $F_u$ is an immersion whenever $u$ is sufficiently small in $C^2$ and we refer to the mapping $u \mapsto H_{F}(u)$ as the mean curvature operator at $F$. The  variation of the mean curvature of $F$ induced by a vector field $V$ is given by
\begin{align}\label{FirstVariation}
\left. D H \right|_{F} (V) : =  \nabla_{V^\top} H + L_F V^\perp
\end{align}
 wheret $V^\perp = V \cdot N_F$ denotes the normal part of $V$ and $V^{\top} = V - V^\perp N_F$ the tangential part, and where $L_F$ is the stability, or jacobi, operator of $F$ and is given by
 \[
 L_F = \Delta_{F} + \left| A_F\right|^2.
 \]
 where above $\Delta_F$ denotes the Laplace operator of the immersion $F$ and $\left| A_F\right|^2$ the square length of the second fundamental form of $F$.

\subsection{The Poisson-Dirichlet Problem for the mean curvature operator }
 
 If $\mc{Q}$ is a second order elliptic operator, a solution to the $(E, f)$-Poisson-Dirichlet problem for and operator $\mc{Q}$  over the domain $D$, where $E$ is a function on $D$ and $f$ is a function on $\partial D$, is  a function $u$ on $D$ satisfying
\[
\mc{Q}_F (u)=  E, \quad \mc{tr}(u) = f
\]
where above $\mc{tr} (u)$ denotes the trace of $u$. We refer to $E$ as the Poission data and to $f$ as the Dirichlet data of the problem. We assume that $E$ is  class $C^{0, \alpha}$ and $f$ is class $C^{2, \alpha}$. Standard elliptic regularity theory then implies that any solution $u$ is class $C^{2, \alpha}$. The solvability of this equation for small $E$ and $f$ is determined studying the Poisson Dirichlet problem for the linearization $\dot{\mc{Q}}$ of $\mc{Q}$ at $u = 0$: 
\[
\dot{\mc{Q}}_F (u)=  E, \quad \mc{tr}(u) = f
\]
We take $\mc{Q} = H_{F}$ for some fixed immersion $F \in R^{2, \alpha} (D)$, so that $\dot{\mc{Q}} = L_F$.   $L_F$ is a self-adjoint elliptic operator and the linearized  Possion-Dirichlet problem for $L_F$ is solvable if and only if the immersion $F$ is non-degenerate, meaning that the kernel of $L_F$ is trivial. 

\subsection{Conformal parametrization of the catenoid}\label{Sec:NormalizedProblem}
The standard catenoid $\mc{C}$ is conformally parametrized by the mapping $\mc{F}: \Omega \rightarrow \mb{R}^3$, where $\Omega : = \mb{S}^1 \times \mb{R}$ is the cylinder and where
\begin{align}\label{Def:Cat}
\mc{F}(x, s) : = \cosh(s) e_r(x) + s e_z, \quad e_r(x) = \cos(x) e_x + \sin(x) e_y.
\end{align}
The first and second fundamental forms of $\mc{F}$ are given by
\[
\mc{g}_{\mc{F}}: = \begin{pmatrix}  \cosh^{2} (s) & 0 \\ \\ 0 & \cosh^{2} (s)\end{pmatrix},\quad A_{\mc{F}} : = \begin{pmatrix}  
1 & 0 \\ \\ 0 & 1
\end{pmatrix}.
\]
The unit normal is 
\begin{align}\label{CatUnitNormal}
N_{\mc{F}} = \cosh^{-1}(s) e_r(x) - \tanh(s) e_z.
\end{align}
and the stability operator is 
\[
L_{\mc{F}} = \cosh^{-2} (s) \Delta_{\Omega} + 2 \cosh^{-4} (s).
\]

\subsection{Catenoidal necks}
For each $\tau > 0$, let $\mc{G}_{\tau}: \Omega \rightarrow \mb{R}^3$ denote the mapping given by
\[
\mc{G}_{\tau} : = \tau \mc{F}.
\]
The neck $\mc{N}_\tau$ is parametrized by the restriction of $ \mc{G}_{\tau}$ to the domain 
\[
\mc{N}_{\tau}^* : = \Omega_{l_{\tau}} : = \mb{S}^1 \times [- l_\tau, l_\tau], \quad l_\tau = \text{Arccosh} \left(\frac{1}{\tau}\right).
\]
Given constants $\tau_0$ and $\tau$, we define a family of diffeomorphisms $\phi_{\tau_0}^{\tau}: \mc{N}_{\tau_0} \rightarrow \mc{N}_{\tau}$ as follows:  First, observe that there is a scalar $\lambda$ such that $\lambda \mc{N}_{\tau}$ is a normal graph over $\mc{N}_{\tau_0}$. The scalar $\lambda$ can be determined as follows: The boundary of the surface $\mc{N}_{\tau}$ in cylindrical coordinates is given by
\[
(r, z)  = \left(1,  \tau l\right)
\]
 and outward conormal is $\left( \sqrt{1 - \tau^2}, \tau \right)$. We choose the scalar $\lambda$ so that the difference vector between the boundaries is orthogonal to the outward conormal, which by rotational symmetry implies that it is a multiple of the unit normal. This means that 
\[
\left\{\lambda(1, \tau l) - (1, \tau_0 l_0)  \right\}\cdot \left( \sqrt{1 - \tau^2}, \tau \right) = 0.
\]
Solving for $\lambda$ gives
\[
\lambda = \lambda_\tau = \frac{\tau^2_0 l_0 +\sqrt{1 - \tau_0^2}}{\sqrt{1 - \tau_0^2} + \tau_0 \tau l}
\]
Observe that $\lambda = 1$ when $\tau = \tau_0$ and 
\[
\dot{\lambda} : = \left. \frac{d}{d \tau} \right|_{\tau = \tau_0} \lambda \rightarrow 0
\]
 as $\tau_0 \rightarrow 0$.  The diffeomorphisms $\phi_{\tau_0}^\tau: \mc{N}_{\tau_0} \rightarrow \mc{N}_{\tau}$ induce  diffeomorphims on the domains $\phi_{\tau_0}^\tau: \mc{N}^*_{\tau_0} \rightarrow \mc{N}^*_{\tau}$, which we still denote by $\phi_{\tau_0}^{\tau}$, which satisfy 
 \[
 \lambda \mc{G}_{\tau_0}^\tau - \mc{G}_{\tau_0} = f N_{\mc{G}_{\tau}}
 \]
 for  a function $f$, where above we have set $\mc{G}_{\tau_0}^{\tau} = \mc{G}_{\tau} \circ \phi_{\tau_0}^\tau$. Differentiating in $\tau$  at $\tau = \tau_0$ and using that $\phi_{\tau_0}^\tau$ is a function of $s$ only  gives:
 \begin{align}\label{CatParamDif}
\left( \dot{\lambda} \tau  + 1\right) \mc{F} - \frac{\partial \mc{G}_{\tau}}{\partial s} \dot{\phi} = \dot{f}N_{\mc{G}_{\tau}}.
 \end{align}
 Thus, the normal  variation field generated by the family of immersion $\mc{G}_{\tau_0}^\tau$ is
 \[
 \left(\dot{\mc{G}}_{\tau_0}\right)^{\perp}  = \left(\tau \dot{\lambda} + 1 \right)\mc{F}^\perp =  \left(\tau \dot{\lambda} + 1 \right)\xi,
 \]
 where above $\xi = s \tanh(s) -1$.  Here, we have used that $\mc{N}_{\mc{G}_\tau} = \mc{N}_{\mc{F}}$ and (\ref{Def:Cat}) and (\ref{CatUnitNormal}). Taking the dot product of (\ref{CatParamDif}) with $\frac{\partial \mc{G}_{\tau_0}}{\partial s}$ gives
 \begin{align}\label{DotPhiExpression}
 \dot{\phi}_{\tau_0} =  \left( \tau \dot{\lambda} + 1\right)\frac{\mc{F} \cdot \frac{\partial \mc{G}_{\tau_0}}{\partial s}}{\left| \frac{\partial \mc{G}_{\tau_0}}{\partial s}\right|^2} =  \left( \tau \dot{\lambda} + 1\right)\frac{\tanh(s) + s/\cosh^{2} (s)}{\tau_0}.
 \end{align}
The total variation field is
 \begin{align}\label{GtotalVariation}
 \dot{\mc{G}}_{\tau_0} = \left( \tau \dot{\lambda} + 1\right) \xi N_{\mc{F}} - \dot{\lambda} \tau \mc{F}.
 \end{align}
\subsection{The modified Dirichlet-Poisson Problem on catenoidal necks}
 As mentioned above,  there are obstructions to solving the  Poission-Dirichlet problem  for $H$ on catenoidal necks with uniform estimates for small $\tau$. We instead impose the additional condition that the solutions are \emph{normalized}--defined below--and in exchange we give up control on the lower modes of the boundary values.  A lower mode on the circle is a function in the span of $1$, $\cos(\theta)$ and $\sin(\theta)$. If $f$ is a function on the circle $\mb{S}^1$, we let $\ol{f}$ denote the projection of $f$ onto the space of lower modes and $\mr{f}$ denote the projection of $f$ into the orthogonal compliment. We refer to $\ol{f}$ as the lower part of $f$ and to $\ol{f}$ as the higher part. More generally, if $E$ is a function on a cylindrical domain $\Omega_{a, b}$ for some $a, b$, we define the lower part of $E$ to be the function $\ol{E} (x, s)$, where for each $s$ $\ol{E}(x, s)$ is  the lower part of the function  $x \mapsto E(x, s)$. The higher part of $E$ is then
\[
\mr{E}: = E - \ol{E} 
\]
and we refer to $\mr{E}$ as the \emph{higher part} of $E$. 
 
  A solution to the \emph{modified} $(E, f)$-Poisson Dirichlet problem on $\Omega_l$ for $H_F$ is then a normalized function  $u$ satisfying
 \[
 H_{F}(u) = E, \quad \mr{\mc{tr}}(u) = \mr{f}. 
 \]
Observe here that only the higher part $\mr{\mc{tr}} (u)$ is equal to $\mr{f}$ and thus 
\[
\mc{tr}(u) = f + \ell
\]
for some lower mode $\ell$. The  boundary $\partial \Omega_{l}$ of $\Omega_l$ comprises the two components $\partial_{\pm} \Omega_{l} : = \mb{S}^1 \times \{\pm l \}$, each of which is canonically isomorphic to $\mb{S}^1$. A function $f$ on $\partial \Omega$ may then be identified with a pair of functions on $\mb{S}^1$ under the identification $f \mapsto (f_-, f_+)$, where $f_{\pm}(x) = f (x, \pm l)$ denotes the restriction of $f$ to $\partial_{\pm} \Omega$. A function on  $\partial \Omega_{l}$ is then a \emph{lower mode} if its restrictions $f_{\pm} (x)$ are lower modes on $\mb{S}^1$. Thus, the space of lower modes on $\Omega_{l}$ is six dimensional. We will denote by $\mr{C}^{k, \alpha} (\mb{S}^1)$ and  $\mr{C}^{k, \alpha}(\partial \Omega_{l})$ the space of $C^{k, \alpha}$ functions on $\mb{S}^1$ and  $\Omega_{l}$ orthogonal to the lower modes. 

\subsection{Geometric jacobi fields and normalized functions}
The space of lower modes on $\partial \mc{N}_{\tau}$ is generated by the traces of the six dimensional family of jacobi fields generated by translations, rotatations, and dilations of $\mb{R}^3$ which we call geometric jacobi fields.  Since $\mc{F}$ is minimal, we have $H_\mc{F} = 0$ and thus by (\ref{FirstVariation}) we have
 \[
 \left. D H \right|_{\mc{F}} (V) =L_\mc{F} V^\perp
 \]
 for any vector field $V$. Since $H$ is invariant under translations and rotations, the left hand side above vanishes for vector fields in the span of the translations $e_x$, $e_y$, and $e_z$, as well as the rotational  vector field $x e_y - y e_x $,  $x e_z - z e_x $ and $y e_z - z e_y$. Finally, since $\mc{F}$ is minimal again,  the mean curvature of $\mc{F}$ is invariant under dilations and thus the left hand side above vanishes additionally when $V$ is the dialtion field $x e_x + y e_y + z e_z$. The normal parts of these vector fields are smooth jacobi fields. They can be computed explicitly from the expressions for $\mc{F}$ and its unit normal in  (\ref{Def:Cat}) (\ref{CatUnitNormal}) as $\tanh(s) $, $\cos(x)\cosh^{-1} (s)$, $\sin(x)\cosh^{-1} (s)$,  $s \tanh(s) - 1$, $\cos(x)\left(\sinh(s) + s\cosh^{-1}(s)\right)$ and $\sin(x)\left(\sinh(s) + s\cosh^{-1}(s)\right)$.  They are precisely the six dimensional space of functions such that $L_{\mc{F}} \mc{j} = 0$ and $\mr{\mc{j}} = 0$ on $\Omega$. Observe that, although  there are seven vector fields, the normal part of the rotational field $x e_y - y e_x$ vanishes on $\mc{F}$, and so the space of geometric jacobi fields is six dimensional. Moreover,  it can be verified directly that the trace map restricts to an isomorphism from the space of geometric jacobi fields onto the space of lower modes on $\mc{N}_{\tau}$. Geometric jacobi fields are uniquely determined by their Cauchy data along  the curve $\{ s = 0\}$, which we call the \emph{signature}. More precisely, if $u$ is a function on $\Omega_{l}$, we define its signature to be the pair $\mc{s}(u) : = (\ol{u}, \ol{u}_s)$, where $\ol{u}$ and $\ol{u}_{s}$ denote the projection of $u (x, 0)$ and $\frac{\partial u}{\partial s}(x, 0)$ onto the space of lower modes.  The function $u$ is said to be \emph{normalized} if it has vanishing signature. It is straightforward to verify that the mapping $u \mapsto \mc{s} (u)$ is an isomorphism of  the space of geometric jacobi fields onto the space of signatures--i.e. pairs of lower modes. Thus,  there is for each function $u$ on $\Omega_l$ a unique geometric jacobi field $\mc{j}$ such that $u + \mc{j}$ is normalized. 
  
 \subsection{Local compactness of the catenoid modulo scalings}
 The parametrization $\mc{F}$ induces an atlas on the catenoid which is compact modulo translations and scalings, which  allows us to prove uniform estimates for the mean curvature operator at $\mc{F}$ in the weighted holder spaces $W^{k, \alpha, \gamma}_{\tau}$. Namely, if $u: \Omega_{a, b} \rightarrow \mb{R}^k$ is a function defined on the finite cylinder $\Omega_{a, b} : = \mb{S}^1 \times [a, b]$, we define the \emph{localization of $u$ at $s_0$} for $a + 1/2 \leq s_0 \leq b - 1/2$ to be the mapping $\left. u \right|_{s_0}: \Omega_0 \rightarrow \mb{R}^3$, $\Omega_0 : = \Omega_{-1/2 ,1/2}$ given by
\[
\left. u \right|_{s_0} (x, s) : = u(x, s + s_0).
\]
The localizations $\left. \mc{F} \right|_{s_0}$ of $\mc{F}$ are  compact in the moduli space--i.e. up to scalings and rigid motions. In particular, if we define the normalization of $\left. \mc{F} \right|_{s_0}$ to be the mapping  $\left. \wt{\mc{F}} \right|_{s_0} : \Omega_0 \rightarrow \mb{R}^3$ given by
\[
\left. \wt{\mc{F}} \right|_{s_0}  = \left(\left. \mc{F} \right|_{s_0} - s_0 e_z\right)/\cosh(s_0),
\]
which differs from $\left. \mc{F} \right|_{s_0}$ by a translation and scaling, then explicitly we have
 \[
 \left. \wt{\mc{F}} \right|_{s_0} = \mc{c}_{s_0} (s) e_r(x) + \frac{s}{\cosh(s_0)} e_z,
 \]
where above we have set 
\[
\mc{c}_{s_0} (s) = \frac{\cosh(s + s_0)}{\cosh(s_0)} = \cosh(s) + \tanh(s_0) \sinh(s).
\]
Thus, as $s_0 \rightarrow \infty$, the normalized localizations  $\left. \wt{\mc{F}} \right|_{s_0}$ converge in $C^{k}(\Omega_0, \mb{R}^3)$ to the limiting map 
\[
\left. \wt{\mc{F}} \right|_{\infty}(x, s) : = e^s e_r(x)
\]
conformally parametrizing a flat annulus in the plane.

\subsection{The weighted mean curvature operator}
 We exploit the  compactness of the normalized charts $\left. \wt{\mc{F}} \right|_{s_0}$ by working  with the weighted mean curvature operator $H_{\mc{F}}'$ at $\mc{F}$ defined below:
\[
H'_{\mc{F}}(u): = \cosh(s) H_F(\cosh(s) u) 
\]
 Localizing $H'_{\mc{F}}(u)$ at $s_0$ and using that the mean curvature operator is homogeneous degree $-1$ gives
 \begin{align}\label{WeightedOperatorScaling} 
 \left. H'_{\mc{F}}(u) \right|_{s_0}&  = \left. \cosh(s) H_{\mc{F}} (\cosh(s) u) \right|_{s_0} \\
 & = \mc{c}_{s_0} H_{\left. \wt{\mc{F}}\right|_{s_0}}\left( \mc{c}_{s_0} \left. u \right|_{s_0} \right) \notag
 \end{align}
 A function $u$ is a solution to modified $(E, f)$-Poission-Dirichlet problem for $H_\mc{F}$ if and only if $u/\cosh(s)$ is a solution to the normalized $(E \cosh(s), f)$ Dirichlet-Possion problem for $H'_{\mc{F}}$. The linearization $\mc{L}'_{\mc{F}}$ of $H'_{\mc{F}}$ at 0 is then 
\begin{align} \label{LPrimeOperator}
L'_{\mc{F}}u& = \cosh^{-1} (s) \Delta \left( \cosh(s) u\right) + 2 \cosh^{-2} (s) u.
\end{align}
Observe that $L'_{\mc{F}}$ is conjugate the operator
\begin{align}\label{L2PrimeOperator}
L^{''}_{\mc{F}} = \Delta  + 2 \cosh^{-2}(s)
\end{align}
 
\subsection{Weighted Holder spaces}
 Given $\tau > 0$, set  
\[
\omega_\tau(s) = \tau  \cosh(s), \quad s \in \mb{R}.
\]
Then, given $\gamma \in \mb{R}$, $\alpha \in (0, 1)$, a non-negative integer $k$ and a locally $C^{2, \alpha}$ function on $\Omega_{l}$ we set
\[
\| u\|^{k, \alpha, \gamma}_{\tau} = \sup_{s_0} \left(\omega_\tau (s_0)\right)^{-\gamma} \left\| \left. u \right|_{s_0}\right\|
\]
where above the supremum is taken over $s_0 \in (-l + 1/2, l - 1/2)$.  We let $W^{k, \alpha, \gamma}_{\tau}$ denote the space $C^{k, \alpha, \gamma}(\Omega_{l_\tau})$ equipped with the norm $\| -\|^{k, \alpha, \gamma}_{\tau}$.

It follows easily from the definition that mapping $\| -\|^{k, \alpha, \gamma}_\tau$ is indeed a norm and that $W^{k, \alpha, \gamma}_{\tau}$ is a Banach space. Observe that $\omega_\tau (s) \leq 1$ for $|s| \leq l_\tau$ and $\gamma > 0$.  It then  follows that 
\begin{align}\label{NormImprovement}
\left\|\left.  u \right|_{s_0}\right\|^{k, \alpha} \leq \left\| u\right\|^{k, \alpha, \gamma}_{\tau} \omega_{\tau} (s_0) \leq \left\| u\right\|^{k, \alpha, \gamma}_{\tau}.
\end{align}
We will use (\ref{NormImprovement}), along with the compactness of the normalized mappings $\left. \wt{\mc{F}} \right|_{s_0}$, in establishing uniform bounds for $H'_\mc{F}$ as map on the spaces $W^{k, \alpha, \gamma}_{\tau}$. We let $W^{k, \alpha, \gamma}_{\tau, n} \subset W^{k, \alpha, \gamma}_{\tau}$ denote the subspace of normalized functions in $W^{k, \alpha, \gamma}_{\tau}$, and we set
\[
\wh{W}^{k, \alpha, \gamma}_{\tau} : = W^{k, \alpha, \gamma}_{\tau} \times \mr{C}^{2, \alpha} \left( \partial \Omega_{l_{\tau}} \right)
\]

\section{Outline of the argument}
We find solutions to the modified Poisson-Dirichlet problem by applying the implicit function theorem. We do this in in two parts. First, we show in Proposition  \ref{Prop:SmoothnessinWeightedSpaces}  that the operator $H'_\mc{F}$ is defined on a ball $\mc{B} = \mc{B}_{\tau} = B_{\epsilon}\left(0;  W^{2, \alpha, \gamma}_\tau \right)$ of fixed radius independent of $\tau$ in the weighted spaces $W^{2, \alpha, \gamma}_\tau$ and that the corresponding Possion-Dirichlet mapping
\[
\wh{H}'_\mc{F}: u \mapsto \left(H'_\mc{F}(u), \mr{\mc{tr}} (u) \right)
\]
is uniformly bounded in $C^{2}$ as a map from $\mc{B}_{\tau}$ into $\wh{W}^{0, \alpha, \gamma}_{\tau}$.  Then, in   Proposition \ref{Prop:WeightedInvertibilty}, we prove that the lineraized mapping
\[
\wh{L}'_\mc{F}: u \mapsto \left(L'_\mc{F}(u), \mr{\mc{tr}} (u) \right)
\]
is an isomorphism from $W^{2, \alpha, \gamma}_{\tau, n}$ onto $\wh{W}^{0, \alpha}_{\tau}$ for small $\tau$ and with a uniform invertibility constant independent of $\tau$.  This allows us to directly apply the implicit function theorem to construct solutions $\mc{h}'_{\tau, E, f} $ to the modified $(E, f)$-Possion-Dirichlet for $H'_\mc{F}$ for data small in $\wh{W}^{0, \alpha, \gamma}_{\tau}$. This is done Proposition \ref{Prop:PDSolutions}.  Corollary \ref{Prop:CMCSolutions} then uses  Proposition \ref{Prop:PDSolutions} to  obtain solutions $\mc{h}_{\tau, \delta, f}$ to the modified  constant mean curvature  Dirichlet problem for the operator $H_\mc{F}$. As a function on $\mc{N}_{\tau}$, $\mc{h}_{\tau, \delta, f}$ satisfies the estimate
\[
\| \mc{h}_{\tau, \delta, f}\| \leq^C \epsilon r^{1 + \gamma}
\]
for a universal constant $C > 0$. Thus, the restrictions  $\mc{h}_{\tau, \delta, f}^{\pm}$ of $\mc{h}_{\tau, \delta, f}$ to $\mc{N}^\pm_{\tau}$ converge on compact subsets of the unit disk $\mc{D}$ way from the origin to a  limit $\mc{h}_{\delta, f_{\pm}}$ at $\tau = 0$, satisfying
\[
H_{\mc{D}}\left( \mc{h} \right) = 0, \quad \mr{\mc{tr}}(\mc{h}) = f_\pm.
\]
and the weighted estimate
\[
\left|\mc{h}(r, \theta) \right| \leq^C \epsilon r^{2}.
\]
Here, the improved estimate for the solution on disk follows from fact that it extends smoothly to the origin. Such solutions are unique (Corollary \ref{Prop:DiskPDProblem}), and thus the solutions $\mc{h}_{\tau, \delta, f}$ extend continuously to $\tau = 0$.
 
The remainder of the article is aimed at establishing  the $C^{1}$ extension of the family of solutions to $\tau = 0$. This is done in several steps.  First, in Lemma \ref{Prop:VaritiationAtZero} we show that the functions $\dot{\mc{h}}^+_{ \delta, f}$ are well defined in $C^2$  and solve a modified Possion-Dirichlet problem on the unit disk. Then  in Proposition \ref{Prop:HDotBounds}, we show that   $\dot{\mc{h}}^+_{\tau_0, \delta, f} \rightarrow \dot{\mc{h}}^+_{\delta, f}$ as $\tau_0 \rightarrow 0$.  The estimates required for this  are delicate and require optimal bounds for the solutions in the weighted spaces. We obtain optimal bounds by comparing  the solutions $\mc{h}^{\pm}_{\tau, \delta, f}$ with their limits $\mc{h}^{\pm}_{\delta, f}$ at $\tau = 0$, which extend smoothly to the unit disk and vanish to second order at the origin.  Since the catenoid  has rapidly decaying curvature, the limiting solutions serve as good approximate solutions on the necks. The error term is then small enough in the weighted space to obtain the improved estimate. This is done in Proposition \ref{Prop:ImprovedDecay}. In Lemma \ref{Prop:TauVariationEquation}, we show that the  $\tau$-variation $\dot{\mc{h}}_{\tau_0, \delta, f}$ satisfies a normalized Poisson-Dirichlet equation with vanishing Dirichlet data. The Poisson data  decomposes as a sum of two terms $\mc{E}_1 + \mc{E}_2$ satisfying uniform estimates in different weighted norms; one large but exponentially decaying  and the second small but exponentially growing.  The mapping properties  for the linearized operator $L'_\mc{F}$ for exponentially decaying source terms is similar to the case of exponentially growing terms belonging to the spaces $W^{k, \alpha, \gamma}_{\tau}$ for positive $\gamma$. However there are some complications caused by the decay related to the presence of small eigenvalues for $L'_\mc{F}$ on long domains $\Omega_{l}$. We handle this by assuming an orthogonality condition of the source term to the lower modes, and directly integrate the low modes. The invertibility theorem for source terms with exponential decay is recorded in Proposition \ref{Prop:DecayInvertibility}. We  prove uniform estimates in Lemma \ref{Eq:HDotTotalBound} for  $\dot{\mc{h}}_{\tau, \delta, f}$. The variations $\dot{\mc{h}}^{\pm}_{\tau, \delta, f}$ then converge to a limit on the disk. To show that the limit agrees with $\dot{\mc{h}}^+_{0, \delta, f}$, we appeal to the  uniqueness of  solutions to the modified Poission-Dirichlet problem on the disk. To apply this theorem, we need improved estimates for the lower modes of term $\mc{E}_1$.  This is proved in Lemma \ref{E1Estimates} and is done by exploiting the rotational symmetry of the catenoid. The solutions then satisfy the uniform estimate
\[
\left|\dot{\mc{h}}_{\tau, \delta, f}\right| \leq^{C\epsilon}  \tau/r^\gamma + r^{1 +\gamma}.
\]
The limit is then normalized and solves the same modified Poisson-Dirichlet problem as $\dot{\mc{h}}^{\pm}_{\delta, f}$ and thus the two solutions agree by Lemma \ref{ModifiedDiskProblem}. Thus the $\tau$-variations extend in continuously to $\tau = 0$.

\section{The mean curvature operator in weighted holder spaces} \label{Sec:MCInWeightedSpaces}
In this section we  record Proposition  \ref{Prop:SmoothnessinWeightedSpaces}, which states that the operator $H'_\mc{F}$ is defined and smooth on a ball of fixed radius independent of $\tau$:

\begin{proposition} \label{Prop:SmoothnessinWeightedSpaces}
There are constants $\epsilon > 0$ and $C > 0$ such that: Given $\tau \in (0, \epsilon)$, the mapping $H'_\mc{F}$ is defined on the ball $\mc{B}: = B_{\epsilon} (0: W^{2, \alpha,\gamma}_{\tau})$  and it holds that 
\[
\left\| H'_{\mc{F}}: C^{2}\left( \mc{B},  W^{0, \alpha,\gamma}_{\tau} \right)\right\| \leq C_0.
\]

\end{proposition}

\begin{proof}
Fix a precompact open set $\mc{O}_0 \subset R^{3, \alpha}(\Omega_0)$ such that the normalized localizations $\left. \wt{\mc{F}}\right|_{s_0}$ of $\mc{F}$ belong to $\mc{O}_0$ for all $s_0 \in \mb{R}$. There are then constants $\epsilon_0 >0$ and $C_0 > 0$ such that  the mapping $H_{-}(-): (F, u) \mapsto H_F(u)$ is defined on $\mc{O}_0 \times \mc{B}_0$ and satisfies the bound:
\begin{align} \label{Eq1}
\left\| H_{-}(-) : C^2\left(\mc{O} \times \mc{B}_0, C^{0, \alpha} (\Omega_0)\right)\right\| \leq C,
\end{align}
where above we have set $\mc{B}_0 = B_{\epsilon_0} (0 : C^{2, \alpha} (\Omega_0))$.  Pick $u \in W^{2, \alpha, \gamma}_\tau$ and assume that $\| u\| \leq \epsilon$ for $\epsilon > 0$ to be determined. By (\ref{WeightedOperatorScaling}) we have
\begin{align*} 
\left. \left. D H'_{\mc{F}} \right|_u(\dot{u}) \right|_{s_0}& =\mc{c}_{s_0} \left. DH_{\left. \tilde{\mc{F}} \right|_{s_0}} \right|_{\mc{c}_{s_0}\left.  u \right|_{s_0}} \left(\left.\mc{c}_{s_0} \dot{u} \right|_{s_0} \right)
\end{align*}
Observe that
\begin{align*}
\left\| \mc{c}_{s_0}  \left. u \right|_{s_0} \right\|^{2, \alpha} \leq^C\left\|  \left. u \right|_{s_0} \right\|^{2, \alpha} \leq C \epsilon.
\end{align*}
Taking $\epsilon$ small and using (\ref{NormImprovement}) and Equation (\ref{Eq1}) gives
\begin{align*}
\left\|\left. \left. D H'_{\mc{F}} \right|_u(\dot{u}) \right|_{s_0} \right\|^{0, \alpha} & \leq^{C} \left\| \mc{c}_{s_0}\left. \dot{u}\right|_{s_0}\right\|^{2, \alpha} \\
& \leq^{C}  \left\|  \left. \dot{u}\right|_{s_0}\right\|^{2, \alpha} \\
&\leq^{C}  \left\|  \left. \dot{u}\right|_{s_0}\right\|^{2, \alpha, \gamma}_{\tau} \omega_{\tau}^{\gamma} (s_0).
\end{align*}
Thus, taking the supremum over $s_0$ in $[- l_{\tau} + 1/2, l_{\tau} -1/2]$ gives
\[
\left\| \left. D H'_{\mc{F}} \right|_u (\dot{u})\right\|^{0, \alpha, \gamma}_\tau \leq^C \| \dot{u} \|^{2, \alpha, \gamma}_\tau
\]
Estimates for the  remaining derivatives are similar and we omit a detailed proof. 
\end{proof}

\section{The stability operator in the spaces $W^{k, \alpha, \gamma}_{\tau}$} \label{Sec:WeightedInvertibility}
In this section we study the mapping properties of the operator $L'_\mc{F}$ as a map between the   weighted  holder spaces $W^{k, \alpha, \gamma}_{\tau}$. The results we present here are standard, and similar statements, along with methods of proof, can be found in \cite{mazzeo-pacard}  (c.f. Proposition 21, \cite{mazzeo-pacard}). The basic idea is to reduce the problem to a classification of Jacobi fields on cylinders with various limitations on their boundary conditions and growth rates. This is done below in Lemma \ref{Prop:Classification}.

\begin{lemma} \label{Prop:Classification} 
The following statements hold:
\begin{enumerate}
\item \label{Prop:Classification1}  Let $u$ be a smooth solution to the equation
\[
L_{\mc{F}}^{''}u = \left( \Delta_{\Omega} + 2 \cosh^{-2}(s)\right)u = 0
\]
on the domain $\Omega$ and assume that
\[
\left| u(x, s)\right| \leq^C \cosh^{1 + \gamma}(s)
\]
for some $\gamma \in (0, 1)$.  Then $u$ is a geometric jacobi field. \\

\item \label{Prop:Classification2} Let $u$ be a smooth solution to the equation 
\[
\Delta_{\Omega} u = 0
\]
on $\Omega$ and assume that 
\[
\left| u(x, s)\right| \leq^C \cosh^{1 + \gamma}(s)
\]
Then $u \equiv 0$.  \\
\item \label{Prop:Classification3} Let $u$ be a smooth solution to the equation 
\[
\Delta_{\Omega} u = 0, \quad \mr{\mc{tr}} (u) = 0.
\]
on $\Omega_- := \mb{S}^1 \times (-\infty, 0]$ and assume that 
\[
\left| u(x, s)\right| \leq^C e^{1 + \gamma}(s)
\]
Then $u \equiv 0$. 
\end{enumerate}
\end{lemma}.
\begin{proof}
Let $u$ be as in statement (\ref{Prop:Classification1}), and set
\begin{align} \label{Eq:KthMode}
u_{k} (s) : = \int_{- \pi}^\pi u(x, s) e^{i k x} dx.
\end{align}
Then $u_k$ satisfies the equation
\begin{align} \label{Eq:KthModeEquation}
u_k''  = \left(k^2 -2 \cosh^{-2} (s)\right) u_k,
\end{align}
and by assumption we have 
\begin{align} \label{Eq:UkGrowthBound}
|u_k(s)| \leq^C \cosh^{1 + \gamma}(s).
\end{align}
 For $k \geq 2$ there is $K >0$ such that 
 \[
 u''_k \geq (1 + \gamma)^2  u_k
 \]
 for $s > K$ and for any $\gamma \in (0, 1)$. Write 
 \[
 f   = a e^{(1 + \gamma) s} + b e^{- (1 +  \gamma) s}
 \]
 for constants $a$ and $b$ to be determined. Clearly it holds that 
 \[
 f'' = (1 + \gamma)^2 f.
 \] 
 Since the matrix
\[
 \begin{pmatrix} e^{(1 + \gamma) K} & e ^{- (1 + \gamma) K} \\ \\ (1 + \gamma) e^{(1 +\gamma) K} & - (1 + \gamma) e^{(1 + \gamma) K}  \end{pmatrix}
 \]
 is non-singular, the values of $f$ and $f'$ at $s = K$ can be prescribed by choosing $a$ and $b$. In particular, with $r =1+ \gamma$ we have
 \begin{align*}
 a & = - \frac{1}{2 r} \left( f(K)  re^{r K} +  f'(K) e^{-r K}\right) \\
  b & = \frac{1}{2 r}\left( - r e^{r K} f(K) + e^{r K}f' (K)\right)
 \end{align*}
 Taking $f(K)  =u (K)$, we can pick $f'(K)$ sufficiently negative such that $a \geq 0$
 . The function $w = u - f$ then satisfies $w (0) =  0$ and $w' (0) >  0$ and the inequality
 \[
 w'' \geq (1 + \gamma)^2 w.
 \]
We claim that $w > 0$ on the interval $[K, \infty)$. If not, there is a smallest time $s_0$ such that $w(s_0) = 0$. This implies that $w$ has a positive maximum on $[K, s_0]$ which contradicts the maximum principle. Thus, $w$ is positive on $[K, \infty)$ and 
\[
w(s) \geq \frac{a}{2} e^{ r s}
\]
which contradicts (\ref{Eq:UkGrowthBound}). This is (\ref{Prop:Classification1})

Now, let $u$ be as in  Proposition \ref{Prop:Classification} (\ref{Prop:Classification2}) and let $u_k$ be as in the proof of  Proposition \ref{Prop:Classification} (\ref{Prop:Classification1}). Then in this case $u_k$ satisfies
\[
u''_{k} = k^2 u_k
\]
and thus 
\begin{align}\label{Eq:ModeExpression}
u_k = a e^{k s} + b e^{-k s}
\end{align}
 for some constants $a$ and $b$, and thus $a$ and $b$ must vanish for $k \geq 0$. This gives  (\ref{Prop:Classification2}).
 
 Now, let $u$ be as in Proposition \ref{Prop:Classification} (\ref{Prop:Classification3}). Then, as in the previous case, we have that $u_k$ is of the form (\ref{Eq:ModeExpression}). Since $|u_{k}(x, s)| \leq^C e^{(1 + \gamma) s}$, we must have that $b = 0$ for all $k$ and $a = 0 $ for $k = 1, 2$. Also, since $\mr{\mc{tr}} (u)  = 0$, we must have $a = 0$ for $k \geq 2$. This proves    (\ref{Prop:Classification3}).
\end{proof}

\begin{lemma}\label{Prop:Non-degeneracy}
The operator $\Delta_{\Omega} + 2 \cosh^{-2} (s)$ has trivial dirichlet kernel on $\Omega_{l}$ for $l \neq \sqrt{2}$
\end{lemma}
\begin{proof}
Let $u$ be a smooth solution to the equation
\[
\Delta u + 2 \cosh^{-2} (s) u = 0, \quad \mc{tr} (u) = 0
\]
on $\Omega_{l}$. Let $u_k$ denote the $k^{th}$ mode of $u$ as in (\ref{Eq:KthMode}). Then $u_k$ satisfies  (\ref{Eq:KthModeEquation}). For $k \geq 2$, it then follows from the fact that $u (-l) = 0$ that $u_k''(s) > 0$ for $s > - \ell$ if $u_k$ does not vanish, which is a contradiction. Thus $u_k \equiv 0 $ for all $k \geq 2$. For the $0$ mode $u_0$, we have 
\[
u_0 = a \tanh(s) + b \left(s\tanh(s) - 1 \right)
\]
for some constants $a$ and $b$. Since $\tanh(s)$ is even and $s \tanh(s) - 1$ is odd and since neither vanishes for $s \neq \sqrt{2}$, it follows that $a = b = 0$. The case of $u_1$ is similar. We have
\[
u_1  = a \left(\cosh^{-1}(s) \right) + b \left(\sinh(s) + s\cosh^{-1}(s)\right)
\]
Since the first of these functions is odd and the second is even, and since niether vanishes, it again follows that $a = b = 0$.
\end{proof} 
We are now ready to prove the invertibility results for the stability operator in the weighted spaces. 

\begin{proposition}\label{Prop:WeightedInvertibilty}
There are constants $\epsilon > 0$ and $L > 0$ such that: Given $\tau \in (0, \epsilon)$,  the mapping $u \mapsto (L'_{\mc{F}} u, \mr{\mc{tr}} (u))$ is an isomorphism from $W^{2, \alpha, \gamma}_{\tau, n}$ onto $\wh{W}^{0, \alpha, \gamma}_{\tau}$  with invertibility constant $L$.
\end{proposition}

We also record here an invertibility theorem  for source terms with vanishing lower modes,  which is needed in the proof of Proposition \ref{Prop:HDotBounds}.

\begin{proposition}\label{Prop:DecayInvertibility}
There is $\epsilon> 0$ such that: Given $\tau \in (0, \epsilon)$, $\gamma \in (-1, 1)$  and a function $E \in C^{0, \alpha} (\mc{N}_{\tau}^*)$  with $\ol{E} = 0$, there is a unique function $u \in C^{2, \alpha}_{0} (\mc{N}_{\tau} )$ solving the $E$-Possion problem for the operator $L'_{\mc{F}}$ on $\mc{N}_{\tau}$. Moreover, there is a  uniform  constant $C> 0$  such that 
\[
\left\| u \right\|^{2, \alpha,  \gamma}_{1} \leq^{C} \left\| E\right\|^{0, \alpha, \gamma}_{1}
\]
\end{proposition}
We prove Propositions \ref{Prop:WeightedInvertibilty} and \ref{Prop:DecayInvertibility} together. 

\begin{proof}[Proof of Propositions \ref{Prop:WeightedInvertibilty} and \ref{Prop:DecayInvertibility}]
We claim first that there  is a  constant $C > 0$ such that 
\[
\left\| u\right\|^{2, \alpha, \gamma}_\tau \leq^C \left\| L'_{\mc{F}} (u)\right\|^{0, \alpha, \gamma}_{\tau} + \| \mr{\mc{tr}} (u)\|^{2, \alpha}
\]
for all $\tau $ sufficiently small. If not, there is a sequence $\tau_j \rightarrow 0$ and  $u_j \in W^{2, \alpha, \gamma}_{\tau_j}$  with 
$\| u_j\| = 1$ and $ \left\| L'_{\mc{F}} (u_j)\right\| + \| \mr{\mc{tr}} (u_j)\| \rightarrow 0$.  Pick $s_j$ such that 
\[
1 = \| u_j\| = \omega_{\tau_j}^{- \gamma} (s_j) \left\| \left. u_j \right|_{s_j}\right\|,
\]
and set
\begin{align} \label{Def:RenormalizedSolutions}
\tilde{u}_j: = \frac{u_j}{\omega_{\tau_j}(s_j)}
\end{align}
Then it holds that 
\begin{align*}
\left\|\left. L_j \tilde{u}_j \right|_{s_0}  \right\| & \leq^C\left( \frac{\omega_{\tau_j}(s_0 + s_j)}{ \omega_{\tau_j}(s_j)} \right)^{\gamma} \\
& \leq^C \cosh_{s_j} (s_0).
\end{align*}
where we have set
\[
L_{j}u   =  \left. L'_\mc{F} \right|_{s_j} u: =  \Delta u +u + 2 \tanh(s + s_j) \frac{\partial  u}{\partial s} + 2 \cosh^{-2}(s + s_j) u.
\]
We may assume without loss of generality that $s_j \geq 0$. Assume first that $s_j$ is convergent to $s_\infty \in \mb{R}$.  Then the sequence $u_j$ converges in $C^{2,\alpha/2}$ to a non-zero limiting function $u$ satisfying
\[
 L'_{\mc{F}}  u = 0
\]
on $\Omega$ and the estimate
\[
\left| u (x, s)\right| \leq \cosh^{\gamma} (s)
\]
Equivalently, setting $v = \cosh(s) u$, then $v$ satisfies
\[
L^{''}_{\mc{F}} v= 0
\]
on $\Omega$ and the estimate
\[
\left| u (x, s)\right| \leq \cosh^{1 + \gamma} (s)
\]
By Lemma \ref{Prop:Classification} (\ref{Prop:Classification1}), the only such functions are the geometric jacobi fields. However, since $u$ is the limit of normalized functions, it is also normalized, which is a contradiction. Now assume that both $s_j$ and $l_{\tau_j} - s_j$ are properly divergent. Then the sequence $\tilde{u}_j$ converges to limiting function $u$ satisfying
\[
\left. \mc{L}'_{\mc{F}} \right|_{\infty}u = \Delta u + u + 2 u_s = e^{-s} \Delta e^{s}(u) = 0 .
\]
on $\Omega$ and the growth estimate 
\begin{align} \label{Ex:GrowthEstimateLimit}
\sup |u(x, s)| \leq^C e^{\gamma s}.
\end{align}
Arguing as in the previous case we obtain a contradiction to Lemma \ref{Prop:Classification} (\ref{Prop:Classification2}). Finally if the sequence $l_{\tau_j} - s_j$ remains bounded, we obtain a non-trivial limit function $u$ satifying
\[
\Delta u = 0, \quad \mr{\mc{tr}}(u) = 0
\]
on the half cylinder $\Omega_-$ and the growth estimate (\ref{Ex:GrowthEstimateLimit}), which is a contradiction to Lemma  \ref{Prop:Classification} (\ref{Prop:Classification3}).  This completes the proof of the a-priori estimate for solutions. To prove existence, first observe that for $l \neq \sqrt{2}$, using the fact that $L'_{\mc{F}}$ and $L''_{\mc{F}}$ are conjugate,  there is by Lemma \ref{Prop:Non-degeneracy} a unique solution $u_0 \in C^{2, \alpha} (\Omega_{l})$ to the $(E, f)$-Poisson Dirichlet problem for $L'_F$ on $\Omega_{l}$:   
\[
L'_{\mc{F}} u_0 = E, \quad \mc{tr} (u_0) = f.
\]
Recalling the discussion in Section \ref{Sec:NormalizedProblem}, there is a unique geometric jacobi field $\mc{j}$ such that $u : = u_0 + \mc{j}$ is normalized. Clearly $u$ is normalized and solves the modified $(E, f)$-Dirichlet-Problem for $L'_{\mc{F}}$ on $\Omega_{l}$. Combing the existence with the a-priori estimate  completes the proof of Proposition \ref{Prop:DecayInvertibility}.

The proof Proposition  \ref{Prop:DecayInvertibility} follows similarly. The existence of a solution $u$ to the problem follows again from Lemma \ref{Prop:Non-degeneracy}. It follows easily that  $u$  inherits  the orthogonality imposed on $E$. To see this, observe that the Fourier expansion of $u$ along meridians satisfies a homogeneous second order ODE with vanishing boundary conditions. We then have
\begin{align*}
\int_{-\pi}^\pi L'_{\mc{F}} & u = \int_{- \pi}^\pi\left( u_{xx} + u_{s s} + 2(\tanh(s) + \cosh^{-2} (s)) u     \right) \\
& = \left(\int u dx\right)_{ss} + 2 (\tanh(s)+ \cosh^{-2}(s)) \left(\int_{- \pi}^\pi u dx \right) \\
& = L'_{\mc{F}} \bar{u}
\end{align*}
where we have set $\bar{u} = \int_{-\pi}^{\pi} u dx$. The orthogonality on $E$ then implies that $\ol{u}$ satifisfies the homogenous equation $L'_{\mc{F}} \ol{u} = 0$ with vanishing boundary conditions. By Lemma \ref{Prop:Non-degeneracy} it follows that $\ol{u} \equiv 0$.  Thus, the modes $u_0$ and $u_1$ vanish. Now, as in the proof of of Proposition \ref{Prop:WeightedInvertibilty}, assume that there is a sequence $l_j \rightarrow \infty$  and functions $u_j$ on $\Omega_{l_j}$ with 
\[
\left\| u_j\right\|^{2, \alpha, \gamma}_1 = 1, \quad \left\| L'_F u_j\right\|^{0, \alpha, \gamma}_1 \rightarrow 0. 
\]
As above, pick $s_j$ such that 
\[
1 = \left\| u_j\right\|^{2, \alpha, \gamma}_1 = \cosh^\gamma (s_j) \left\| \left.u_j \right|_{s_j} \right\|
\]
and define $\tilde{u}_j$ as in (\ref{Def:RenormalizedSolutions}). We consider the same three possibilities as before. In the case that $s_j$ is convergent, we obtain a non-trivial limiting function $u$ satisfying
\[
\Delta_{\Omega}u + 2\cosh^{-2}(s) u = 0, \quad \left| u\right| \leq \cosh^{1 + \gamma} (s_0).
\]
Thus, by Lemma \ref{Prop:Classification} (\ref{Prop:Classification1}) is a geometric jacobi field. However, since $u$ is orthogonal to lower modes along each meridian, this is a contradiction. Assuming now that $l_j - s_j$ and $s_j$ are both divergent, we obtain as a nontrivial limit a function $u$ satisfying
\[
\Delta_\Omega u = 0, \quad \left| u(x, s)\right| \leq \cosh^{1 + \gamma} (s)
\]
which again by  Lemma \ref{Prop:Classification} (\ref{Prop:Classification2}) is a contradiction. Finally, when $l_j - s_j$ is convergent, we obtain as a limit  a function $u$ satisfying
\[
\Delta u = 0, \quad \quad \left| u(x, s)\right| \leq \cosh^{1 + \gamma} (s)
\] 
on $\Omega_-$ and the boundary condition $u(x, 0) =0$. $u$ is then a multiple of the coordinate function $s$. However, this again contradicts the fact that $u$ is orthogonal to lower modes along the meridians.  This completes the proof of Proposition \ref{Prop:DecayInvertibility}.
\end{proof}

\section{Solutions to the non-linear Poisson-Dirichlet problem } \label{Sec:PDSolutions}

As a direct consequence of Propositions \ref{Prop:SmoothnessinWeightedSpaces}, \ref{Prop:WeightedInvertibilty} and the Implicit Function Theorem, we can solve the non-linear modified Poission-Dirichlet Problem for $H_\mc{F}$  for small data.

\begin{proposition}\label{Prop:PDSolutions}
There are constants $\epsilon > 0$ and $C > 0$ such that:

\begin{enumerate}
\item \label{Prop:Solutions1}  Given $(E, f) \in \wh{W}^{0, \alpha, \gamma}_{\tau}$ with $\left\| (E, f) \right\| \leq \epsilon$, there is a unique function $\mc{h} = \mc{h}'_{\tau, E, f} \in W^{2, \alpha, \gamma}_{\tau, n}$ with 
\[
\| \mc{h}'_{\tau, E, f}\|^{2, \alpha, \gamma}_{\tau} \leq^C \| f\|  + \| E \|
\]
 and such that 
\[
H'_{\mc{F}} (\mc{h}'_{E, f}) = E, \quad \mr{\mc{tr}}(\mc{h}'_{E, f}) = f.
\] 
\item \label{Prop:Solutions2} The mapping $u_{\tau, \cdot}(E, f) \mapsto u_{\tau, (E, f)}$ is smooth  from $\mc{B} : = B_{\epsilon} \left( 0, \wh{W}^{0, \alpha, \gamma}_{\tau}\right)$ into $ W^{2, \alpha, \gamma}_{\tau, n}$ and it holds that  
\[
\left\| u_{\tau, \cdot}: C^{k}\left(\mc{B},   W^{2, \alpha, \gamma}_{\tau, n} \right)\right\| \leq C.
\]
\end{enumerate}
\end{proposition}

\begin{proof}
This is a direct consequence of the Implicit Function Theorem.
\end{proof}

The following Proposition reinterprets the solutions constructed in Proposition \ref{Prop:PDSolutions} as solutions for the constant mean curvature Dirichlet problem  for $H_{\mc{F}}$.

\begin{corollary}\label{Prop:CMCSolutions}
There are constants $\epsilon >0$ and $C > 0$ such that: Given $\delta \in (- \epsilon, \epsilon)$, and $f $ with $\| f\| < \epsilon$, there is a unique function $ \mc{h}'_{\tau, \delta, f} \in W^{2, \alpha, \gamma}_{\tau}$ with 
\[
H'_\mc{F}(\mc{h}'_{\tau, \delta, f}) = \delta \cosh(s), \quad \mr{\mc{tr}}(\mc{h}'_{\tau, \delta, f}) = f.
\]
and satisfying the estimate $\| \mc{h}'_{f, \delta, \tau}\| \leq^C \delta + \| f\|$.
\end{corollary}

\begin{proof}
Given $\delta > 0$  and $\tau > 0$,  set 
\[
E_{\tau, \delta} =  \tau\delta \cosh(s), \quad l_\tau : =  \text{Arcsosh} \left(\frac{1}{\tau}\right).
\]
The mapping $\delta \mapsto E_{\tau , \delta} = \delta \omega_{\tau}$ is linear and uniformly bounded from $\mb{R}$ into  $W^{0, \alpha, \gamma}_{\tau}$ independent of $\tau$.  Thus, by Proposition \ref{Prop:PDSolutions}, there is $\epsilon > 0$ such that $\mc{h}'_{\tau, \delta, f} :  = \mc{h}'_{\tau, E_{\tau, \delta}, f} \in W^{2,\alpha,\gamma}_{\tau}$ is defined for $\delta \in (-\epsilon, \epsilon)$ and $\| f\| \leq \epsilon$. By Proposition \ref{Prop:PDSolutions} (\ref{Prop:Solutions2}), the mapping $\mc{h}'_{\tau}: (\delta, f) \mapsto \mc{h}'_{\tau, \delta, f}$ is smooth  for $\tau$ fixed and uniformly bounded in $C^{k}$ independent of $\tau$ in $(0, \epsilon).$ Moreover, by construction we have
\[
H_\mc{F}'(\mc{h}'_{\tau, \delta, f}) = \tau \delta \cosh(s).
\]
\end{proof}

The functions $\mc{h}_{\tau, \delta, f} : =  \omega_{\tau}\mc{h}'_{\tau, \delta, f}$  then satisfy
\[
H_{\mc{G}_{\tau}} (\mc{h}_{\tau, \delta, f}) = \delta, \quad \mr{\mc{tr}} = f.
\]
In the following section, we show that the solutions extend continuously to $\tau = 0$.

\section{$C^0$ extension to $\tau = 0$}
We will need the following lemma which establishes the uniqueness of the limit solutions on the disk. We define solutions to the modified Poission-Dirichlet problem on the disk as follows: If $\mc{Q}$ is a differential operator on the disk $\mc{D}$, a function $u$ on $\mc{D}$ is a solution to the modified $(E, f)$ Poisson-Dirichlet Problem for $\mc{Q}$ on $\mc{D}$ if 
\[
\mc{Q} u =  E, \quad \mr{\mc{tr}}(u) = \mr{f}
\] 
and if 
\begin{align}\label{DiskNormalized}
u (0) = \nabla u (0) = 0
\end{align}
A function $u$ on the unit disk satisfying the conditions (\ref{DiskNormalized})  is said to be \emph{normalized}.  When $\mc{Q}$ agrees with the Laplace operator $\Delta_{\mc{D}}$, clearly the modified Poission Dirichlet problem is uniquely solvable for $E \in C^{0, \alpha} (\mc{D})$ and $f \in C^{2 , \alpha} (\partial \mc{D})$. In the following, we let $ C_n^{k, \alpha} (\mc{D})$ denote the space of normalized class $C^{k, \alpha}$ functions on $\mc{D}$ and we set $\wh{C}^{0, \alpha} (\mc{D}) : = C^{0, \alpha} (\mc{D}) \times \mr{C}^{2, \alpha} (\mb{S}^1)$. Observe that here we are identifying $\partial \mc{D} = \mb{S}^1$.

\begin{lemma}\label{ModifiedDiskProblem}
The mapping $u \mapsto \left(\Delta u, \mr{\mc{tr}}(u) \right)$ is an isomorphism from $C^{2, \alpha}_{n}\left( \mc{D}\right)$ onto $\wh{C}^{0, \alpha}(\mc{D})$
\end{lemma}
\begin{proof}
Since the disk is Dirichlet non-degenerate, there is a unique function $u_0 \in C^{2, \alpha} (\mc{D})$ with 
\[
\Delta_{\mc{D}} u_0 = E, \quad \mc{tr} (u_0) = f.
\]
The function $\ell$ given by
\[
\ell(x, y) = u_0(0) + \left. \nabla u_0  \right|_{0} \cdot (x, y)
\]
is harmonic and $\mr{\mc{tr}} (u) = 0$. The function $u = u_0 - \ell$ then solves the modified $(E, f)$ Poission-Dirichlet problem on $\mc{D}$ for $\Delta_{\mc{D}}$. Clearly, $u$ is unique. This completes the proof. 
\end{proof}

\begin{corollary}\label{Prop:DiskPDProblem}
There are constants $\epsilon >0$ and $ C> 0$ such that: Given $f \in \mr{C}^{2, \alpha}(\mb{S}^1)$ and $E \in C^{0, \alpha} (\mc{D})$ with $\| f\|, \| E\| \leq \epsilon$, there is a unique function $\mc{h} = \mc{h}_{ f, E} \in C^{2, \alpha} (\mc{D})$ such that 
\[
H_\mc{D}(\mc{h}) = E, \quad \mr{\mc{tr}} (\mc{h}) = f
\] 
and such that $\mc{h} (0) = \left. \nabla \mc{h} \right|_{0} = 0$. 
\end{corollary}

\begin{proof}
The mapping $\wh{H}_{\mc{D}}: u \mapsto (H_{\mc{D}}(u), \mr{\mc{tr}}(u))$ is smooth and defined on a ball near the origin in $C_n^{2, \alpha}(\mc{D})$  and maps into $\wh{C}^{0, \alpha}(\mc{D})$. By  Lemma \ref{ModifiedDiskProblem}, its linearization at $u= 0$ is an isomorphism. It is then a diffeomorphism near zero, which is equivalent to the statement of the corollary 
\end{proof}

\begin{corollary}
The functions $\mc{h}^{\pm}_{\tau, \delta, f}$ converge in $C^{2, \alpha/2}$ on $\mc{D}$ away from the origin to the function $\mc{h}_{\delta, f^+}$ on $\mc{D}$, where $f^+$ denotes the restriction of $f$ to $\partial_+ \Omega_{l_\tau}$.
\end{corollary}
\begin{proof}
By Proposition \ref{Prop:CMCSolutions}, the functions $\mc{h}^{\pm}_{\tau, \delta, f}: \mc{N}^{\pm}_{\tau} \rightarrow \mb{R}$ are uniformly bounded in $C^{2, \alpha}$ and satisfy the bound
\begin{align}\label{Ex:Uwer}
\left|\mc{h}^\pm_{\tau, \delta, f} \right| \leq^C r^{1 + \gamma},
\end{align}
where $\gamma \in (0, 1)$ and are locally bounded in $C^{2, \alpha}$. Taking a sequence $\tau_j \rightarrow 0$, the solutions $\mc{h}_{j} : = \mc{h}^\pm_{\tau_j, \delta, f}$, after possibly passing to a subsequence, converge  on the disk $\mc{D}$ in $C^{2, \alpha/2}$ away from the origin to a function $\mc{h}$ satisfying
\[
H_{\mc{D}} (\mc{h}) = 0, \quad \mr{\mc{tr}}(\mc{h}) = f.
\]
The uniform estimate (\ref{Ex:Uwer}) then implies that $u$ is normalized. By Corollary \ref{Prop:DiskPDProblem}, we then have $\mc{h} = \mc{h}_{\delta, f^+}$. This completes the proof. 
\end{proof}

\section{Improved decay estimates for  solutions} \label{Sec:ImprovedDecay}
In order to establish uniform estimates for the $\tau$-derivative of the solutions, we need optimal decay estimates for the solutions $\mc{h}'_{\tau, \delta, f}$, which a-priori lie in the spaces $W^{k, \alpha, \gamma}_{\tau}$. By comparing these solutions with their limits $\mc{h}^{\pm}_{\delta, f}$ on the disk, which extend smoothly to the origin and vanish to first order, we obtain improved estimates for the decay. For later applications, we will also need to separate out the dominant mode, which is higher, from the rest of $\mc{h}_{\tau, \delta, f}$. This allows improved estimate for the lower part of $\mc{h}_{\tau, \delta, f}$ which is needed in the proof of Proposition \ref{Prop:HDotBounds} and  Lemma  \ref{Prop:VaritiationAtZero}.

\begin{proposition}\label{Prop:ImprovedDecay}
It holds that 
\[
\left\| \mc{h}'_{\tau, \delta, f}\right\|^{2, \alpha, 1}_{\tau} \leq^C \epsilon,  \quad \left\|\ol{\mc{h}}'_{\tau, \delta, f} \right\|^{2, \alpha, \gamma}_{\tau} \leq^C \epsilon \tau.
\]
\end{proposition}

\begin{proof}[Proof of Proposition \ref{Prop:ImprovedDecay}]
Fix $f \in \mr{C}^{2, \alpha} (\partial \mc{N}_{\tau})$ and $\delta  \in \mb{R}$ such that $\mc{h}' : = \mc{h}'_{\tau, \delta, f}$ is defined, and let $f_{\pm} (x) = f (x, \pm l_{\tau})$ denote the restriction of $f$ to the boundary component $\partial_{\pm} \mc{N}_{\tau}$. By Corollary \ref{Prop:DiskPDProblem}, there are constants $\epsilon > 0$ and $C > 0$ such that the function $\mc{h}_+ : = \mc{h}_{\delta, f_+}$ is defined and satisfies 
\[
\| \mc{h}_+\|^{2, \alpha} \leq^C \epsilon
\]
 Since $\mc{h}_+$ vanishes to first order on the disk, the  taylor expansion of $\mc{h}_+$ at the origin gives that
 \[
 \mc{h}_{+} = \mc{f}_+ +  \mc{g}_+
 \]
 where we have set
 \[
\mc{f}_+ : =  a x^2 + b y^2 + c xy 
 \]
  for constants $a$, $b$ and $c$ satisfying $|a|, |b|, |c| \leq^C \epsilon$ and where $\mc{g}_+$ satisfies the estimate:
\begin{align}\label{gestimate}
\sup \left|\mc{g}_+ (r, \theta) \right| \leq^C \epsilon r^3,
\end{align}
where $(r, \theta)$ denotes polar coordinates on $\mc{D}$. Define $\mc{h}^*_{+}$ on $ \mc{N}^*_{\tau}$ as follows: Fix a  smooth  monotonic function $\psi: \mb{R} \rightarrow \mb{R}$  such that $\psi (s) = 0$ for $s \leq 10$ and $\psi(s) = 1$ for $s \geq 11$ and set
\begin{align*}
\mc{h}^*_+(x, s) & = \frac{ \psi(s) \left(\mc{h}_{\delta, f_+} \circ \ol{\mc{G}}_{\tau}(x, s) \right)}{\omega} \\
& =  \frac{ \psi(s) \left(\mc{f}_+\circ \ol{\mc{G}}_{\tau}(x, s) \right)}{\omega}  +  \frac{ \psi(s) \left(\mc{g}_+\circ \ol{\mc{G}}_{\tau}(x, s) \right)}{\omega} \\
& = : \mc{f}^*_+ + \mc{g}^*_+
\end{align*}
where we have set 
\[
\ol{\mc{G}}_{\tau}(x, s) = \tau \cosh(s) e_r (x), \quad \omega= \omega_{\tau}= \tau \cosh(s).
\]
Observe that $\ol{\mc{G}}_{\tau}$ takes the domain $\mc{N}^{*, +}_{\tau} : = \Omega_{0, l_\tau}$ into $\mc{D}$ and that it is a diffeomorphism for $s > 0$. Since we have
\[
\left(\ol{\mc{G}}_{\tau}\right)^* (r) = r \circ \ol{\mc{G}}_{\tau} = \tau \cosh(s) = \omega,
\]
 the function $\mc{f}_+^*$ is then explicitly given by
\begin{align}\label{Eq:PullbackDominantPart}
\mc{f}_+^* =  \omega \psi \left(a \cos^2(x) + g \sin^2(x) + 2 \cos(x) \sin(x) \right).
\end{align}
Additionally, we have the estimate
\begin{align}\label{Eq:RemainderEst}
\sup \left| \mc{g}_+ \right| \leq^C \omega^2.
\end{align}
We can  similarly define a function $\mc{h}^*_-$ by pulling back $\mc{h}_{-\delta, f_-}$  from the disk to the bottom half of the catenoidal neck and we set 
\[
\mc{h}^* : = \mc{h}_+^* + \mc{h}^*_- .
\]
We then have that  $\mc{h}^*$ is defined on $\mc{N}_{\tau}^*$ and it holds that 
\[
\mr{\mc{tr}}(\mc{h}^*) = f = \mr{\mc{tr}} (\mc{h}_{\tau, \delta, f}).
\]
We will  estimate the difference
\begin{align}\label{Eq:ErrorEstimate}
\mc{E} : = \ H'_{F} \left( \mc{h}^*\right) - \delta \tau \cosh(s) 
\end{align}
We  consider only that case that $s_0 \in \mb{R}$ is positive, as the case $s_0 \leq 0$ is treated identically. The estimate then follows from considering two cases separately: In the first, we assume $s_0$ to be near the support of $\nabla \psi$, in which case the function $\mc{h}^*$ is sufficiently small and the structure of the error term is unimportant. In the second case, $s_0$ is far from the support of $\nabla \phi$ and we use the fact that here $\mc{h}^*$ is exactly the pullback of the corresponding solution on the disk. In this region, $\mc{h}^*$ is no longer small, but the distance from the catenoidal neck to the flat disk is, which we exploit.  
We consider now the first case. For $s_0$ belonging to the region $[8, 13]$, we have 
\[
\left\| \left. \mc{h}^* \right|_{s_0} \right\|^{2, \alpha} \leq^C \omega_0 \leq^C \epsilon \tau
\]
and  thus trivially
\[
\left\| \left. L_{\mc{F}}' \left(\mc{h}^* \right) \right|_{s_0} \right\|^{0, \alpha} \leq^C \epsilon \tau. 
\]
For the second case,  we localize at $s_0 \geq 12$--that is, away from the support of $\nabla \psi$. Observe that 
\begin{align} \label{Ex:HApproxEstimate}
\left\|\mc{h}^* \right\|^{2, \alpha, 1}_{\tau} \leq^C \epsilon.
\end{align}
Given $s_0  \geq 12$, we have
\begin{align*}
\left. \wt{\mc{F}} \right|_{s_0} (x, s) & = \mc{c}_{s_0} (s) e_r(x)+ \frac{s}{\cosh(s_0)} e_z\\
&  = : \left. \wt{\ol{\mc{G}}} \right|_{s_0} (x, s) + \frac{s}{\cosh(s_0)} e_z
\end{align*}
where $ \left. \wt{\ol{\mc{G}}} \right|_{s_0}$ is defined implicitly above and thus
\begin{align} \label{Ex:FGEstiamte}
\left\| \left. \wt{\mc{F}} \right|_{s_0}  - \left. \wt{\ol{\mc{G}}}_{\tau} \right|_{s_0}  \right\|^{2, \alpha} \leq^C \cosh^{-1}(s_0).
\end{align}
Since
\[
\cosh(s) H_{\ol{\mc{G}}_\tau}(\tau \cosh(s)\mc{h}^*) = \delta \cosh(s),
\]
we  have 
\begin{align}\label{ErrorTerm1}
\left. \mc{E} \right|_{s_0} : = \mc{c}_{s_0}  \left\{ H_{\left. \wt{\mc{F}} \right|_{s_0}} -  H_{\left. \wt{\ol{\mc{G}}}_{\tau} \right|_{s_0}} \right\}\left(\mc{c}_{s_0} \left.\mc{h}^* \right|_{s_0} \right)
\end{align}
In order to get the estimates we want we will have to decompose the error term into its lower and higher parts. 
Given immersions $F, G \in R^{3, \alpha} (\Omega_0)$ and a function $u \in C^{2, \alpha} (D)$, set
\[
\Phi(F, G, u) : = H_{F} (u) - H_{G}(u)
\]
so that above we have
\[
\left. \mc{E} \right|_{s_0}   = \Phi\left( \left. \wt{\mc{F}} \right|_{s_0},\left. \wt{\ol{\mc{G}}}_{\tau} \right|_{s_0},  \left.\mc{h}^* \right|_{s_0} \right)
\]
Then $\Phi$ is smooth as a map into $C^{0, \alpha} (\Omega_0)$ and defined near $(F, G, u) = \left( \left. \wt{\ol{\mc{G}}}_{\tau} \right|_{s_0} , \left. \wt{\ol{\mc{G}}}_{\tau} \right|_{s_0} , 0\right)$ and vanishes whenever $F = G$. We then have
\begin{align}\label{FormalPhiEstimate}
\left\|\left. \frac{\partial^{k}\Phi}{\partial u^{k}} \right|_{(F, G, u)} \right\| \leq^C \left\| F - G\right\|^{3, \alpha}\| u\|^{2, \alpha}
\end{align}
Since  $\left. \wt{\mc{F}} \right|_{s_0}$ and $\left. \wt{\ol{\mc{G}}}_{\tau} \right|_{s_0}$ are both minimal we have $\Phi\left( \left. \wt{\mc{F}} \right|_{s_0},\left. \wt{\ol{\mc{G}}}_{\tau} \right|_{s_0}, 0\right) = 0$ and thus
\begin{align*}
\left. \mc{E} \right|_{s_0} & =  \left(\left. \frac{\partial \Phi}{\partial u} \right|_{\left( \left. \wt{\mc{F}} \right|_{s_0},\left. \wt{\ol{\mc{G}}}_{\tau} \right|_{s_0}, 0 \right)}\right)\left( \mc{h}^*\right) + R \\
& = \mc{E}' + R
\end{align*}
where $R$ is the first order remainder term. Directly from (\ref{Ex:FGEstiamte}) and  (\ref{FormalPhiEstimate}) we get: 
\[
\left\| R\right\|^{0, \alpha} \leq^{C}\left( \cosh^{-1}(s_0)\right)\left( \left\| \mc{h}^*\right\|^{2, \alpha}\right)^2 \leq^{C\epsilon} \tau \omega.
\]
and 
\[
\left\| \mc{E}'\right\|^{0, \alpha}  \leq^C\left( \cosh^{-1}(s_0)\right) \left\| \mc{h}^*\right\|^{2, \alpha}\leq^{C\epsilon} \tau.
\]
We now estimate $\ol{\mc{E}'}$. To do this,  write
\[
\left(\left. \frac{\partial \Phi}{\partial u} \right|_{\left( \left. \wt{\mc{F}} \right|_{s_0},\left. \wt{\ol{\mc{G}}}_{\tau} \right|_{s_0}, 0 \right)}\right)\left( \mc{h}^*\right)  = \left. D H_{ \left. \wt{\mc{F}} \right|_{s_0}} \right|_{0}\left(\mc{h}^* \right) - \left. D H_{ \left. \wt{\ol{\mc{G}}}_{\tau} \right|_{s_0}} \right|_{0}\left(\mc{h}^* \right)
\]
Since   $ \left. \wt{\mc{F}} \right|_{s_0}$ and $ \left. \wt{\ol{\mc{G}}}_{\tau} \right|_{s_0}$ are rotationally symmetric, the linearizations $\left. D H_{ \left. \wt{\mc{F}} \right|_{s_0}} \right|_{0}$ and   $\left. D H_{ \left. \wt{\ol{\mc{G}}}_{\tau} \right|_{s_0}} \right|_{0}$  are then mode separating (recall Section \ref{Sec:ModePreservation}) and thus
\[
\ol{\left(\left. \frac{\partial \Phi}{\partial u} \right|_{\left( \left. \wt{\mc{F}} \right|_{s_0},\left. \wt{\ol{\mc{G}}}_{\tau} \right|_{s_0}, 0 \right)}\right)\left( \mc{h}^*\right)} = \left(\left. \frac{\partial \Phi}{\partial u} \right|_{\left( \left. \wt{\mc{F}} \right|_{s_0},\left. \wt{\ol{\mc{G}}}_{\tau} \right|_{s_0}, 0 \right)}\right)\left( \ol{\mc{h}^*}\right).
\]
This then gives:
\[
\left\|\ol{\mc{E}'} \right\|^{0, \alpha} \leq^C \left( \cosh^{-1} (s_0)\right)\left(\|\ol{\mc{h}^*} \|^{2, \alpha} \right) \leq^C \tau \omega. 
\]
Thus we can split $\mc{E}$ into the sum of the higher and lower parts $\mc{E} = \ol{\mc{E}} + \mr{\mc{E}}$ where
\[
\left\| \mr{\mc{E}}\right\|^{0, \alpha, 0}_1 \leq^C \tau, \quad \left\| \ol{\mc{E}}\right\|^{0, \alpha, 1}_{\tau} \leq^C \tau
\]
Observe that 
\begin{align*}
H'_{\mc{F}}(\mc{h}') - H'_{\mc{F}}(\mc{h}^*) & = \int_{0}^{1} \left. L'_{\mc{F}} \right|_{\mc{h}(t)}\left( \mc{h}' - \mc{h}^*\right)dt \\
&  = : \mc{L}\left( \mc{h}' - \mc{h}^*\right)
\end{align*}
where above we have set $\mc{h}(t)= (1 - t) \mc{h}' + t \mc{h}^*$ and where the operator $\mc{L}$ is defined implicitly above. Combining with (\ref{ErrorTerm1}) we have
\[
\mc{L}\left(\mc{h}' - \mc{h}^*\right) = \mc{E}.
\]
By Proposition \ref{Prop:DecayInvertibility}, there is a unique function $u$  such that
\[
L_{F}'(u) = \mr{\mc{E}}, \quad \mc{tr}(u) = 0
\]
and satisfying the estimate $\left\| u\right\|^{2, \alpha, 0}_{1} \leq^C \epsilon \tau$. 
Then
\begin{align*}
\mc{L}\left(\mc{h}' - \mc{h}^* - u\right) & = \ol{\mc{E}} + \left(\mc{L} - L'_{\mc{F}} \right)(u).
\end{align*}
We then have
\begin{align*}
\left\|\left.   \left(\mc{L} - L'_{\mc{F}} \right)(u) \right|_{s_0}\right\| &\leq^C\sup\left\| \left.  \mc{h}(t) \right|_{s_0}\right\|\left\| \left. u \right|_{s_0}\right\| \\
& \leq^C \epsilon  \tau \omega^{\gamma}.
\end{align*}
 and thus
 \[
 \left\| \mc{L}\left(\mc{h}' - \mc{h}^* - u\right)\right\|^{0, \alpha, \gamma}_{\tau} \leq^C \epsilon \tau.
 \]
The operator $\mc{L}$ is a easily seen to be  perturbation of $L_{F}'$ in the weighted spaces $W^{k, \alpha, \gamma}_{\tau}$,  and thus by Proposition  \ref{Prop:WeightedInvertibilty}, there is a unique $v \in W^{2, \alpha, \gamma}_{\tau, n}$ such that 
 \[
 \mc{L}\left(\mc{h}' - \mc{h}^* - u - v\right)= 0, \quad  \mr{\mc{tr}}(v) = 0, 
 \]
 and such that $\left\|v\right\|^{2, \alpha, \gamma}_{\tau} \leq^C \tau$.
 In particular, this gives the local estimate
 \begin{align} \label{LimitEstimate}
 \left\|\left.  \left(  \mc{h}' - \mc{h}^*\right) \right|_{s_0}\right\|&  \leq^{C\epsilon} \tau + \tau \omega_0^{\gamma} \\
 &\leq^{C\epsilon} \tau \notag
 \end{align}
We then have
 \begin{align*}
 \left\|\left. \mc{h}' \right|_{s_0} \right\| & \leq  \left\|\left. \mc{h}^* \right|_{s_0} \right\| +   \left\|\left. \mc{h}^* \right|_{s_0} - \left. \mc{h}' \right|_{s_0} \right\| \\
 & \leq^{C \epsilon} \omega_0 + \tau \\
 & \leq^{C \epsilon} \omega_0. 
 \end{align*}
 Taking the supremum over $s_0$ gives 
 \[
 \|\mc{h}^{'} \|^{2, \alpha, \gamma}_{1}\leq^C \epsilon.
 \]
 Finally, the second estimate follow from the fact that $\ol{u} = 0$ and thus
 \begin{align}\label{HHigherEsimate}
 \left\| \ol{\mc{h}'} - \ol{\mc{h}^*}\right\|^{2, \alpha, \gamma}_{\tau} =  \left\| v\right\|^{2, \alpha, \gamma}_{\tau} \leq^C \tau.
 \end{align}
This completes the proof.
\end{proof}
\section{$C^{1}$-extension of the Dirichlet Solutions to scale zero} \label{Sec:C1Extension}

In this section, we prove that the $\tau$-derivative of the solutions $\mc{h}_{\tau, \delta, f}$ extend continuously to $\tau  = 0$.  Recall that $\mc{h}_{\tau, \delta, f}$ solves the modified $(\delta, f)$-Poisson Dirichlet problem for $  \mc{G}_\tau$ on the domain $\mc{N}_{\tau}^* = \Omega_{l_\tau}$, so:
\[
H_{ \mc{G}_\tau} \left( \mc{h}_{\tau, \delta ,f} \right) = \delta, \quad \mr{\mc{tr}} (\mc{h}_{\tau, \delta, f}) = f.
\]

 \begin{proposition} \label{Prop:HDotBounds}
 Let $\dot{\mc{h}}_{\tau_0, \delta, f}^\pm$ denote the restrictions of $\dot{\mc{h}}_{\tau_0, \delta, f}$ to $\mc{N}^{\pm}_{\tau_0}$. 
Then the  functions $\dot{\mc{h}}^\pm = \dot{\mc{h}}^\pm_{\tau_0}$ converge in $C^{2, \alpha}$ away from the origin   to $ \dot{\mc{h}}^+_{\delta, f}$ 
\end{proposition}
 
 The proof of Proposition \ref{Prop:HDotBounds} is long and we thus split it  apart into the following steps. First in Lemma \ref{Prop:VaritiationAtZero}, we show that the Gateaux derivative $\dot{\mc{h}^\pm_{\delta, f}}$ at $\tau = 0$ is well-defined in $C^{2, \alpha}$ and solves a modified Poission-Dirichlet problem on the disk with vanishing Dirichlet data. Next, in Lemma \ref{Prop:TauVariationEquation}, we show that  for positive $\tau$, the derivatives $\dot{\mc{h}}_{\tau, \delta, f}$ solves a modified Poission-Dirichlet problem on $\mc{N}_{\tau}$ with vanishing Dirichlet data. The Poisson data of problem splits into   two parts $\mc{E}_1$ and $\mc{E}_2$, both bounded in different weighted holder spaces. The first, $\mc{E}_1$, is large but decaying, and we can apply the invertibility  result Proposition \ref{Prop:DecayInvertibility} to obtain a normalized solution to the Poisson problem for the higher part $\mr{\mc{E}}_1$ of $\mc{E}_1$. The term $\mc{E}_2$ is small but growing and can be solve for using Proposition \ref{Prop:WeightedInvertibilty}. Solving for the higher part $\ol{\mc{E}}_1$ of $\mc{E}_1$ can then be done using separation of variables. This then allows us to show, in Lemma \ref{Eq:HDotTotalBound}, that the functions $\dot{\mc{h}}_{\tau, \delta, f}$ are uniformly bounded in $C^{2, \alpha}$ away from the origin and thus the restrictions $\dot{\mc{h}}^\pm_{\tau, \delta, f}$ converge to a limit on the disk. The limit then satisfies the same Poisson problem as $\dot{\mc{h}}^\pm_{\delta, f}$, with boundary values agreeing up to lower modes. In order to show that the two functions agree, we need to know that that the limit is normalized. To do this, in Lemma \ref{E1Estimates} we show that higher part $\ol{\mc{E}}_1$ satisfies an improved bound relative to the whole term. 
 
 \begin{lemma} \label{Prop:VaritiationAtZero}
 The function $ \dot{\mc{h}}^\pm_{\delta, f} $  is well-defined away from the origin on $\mc{D}$ and satisfies
 \[
 \left. D H_\mc{D} \right|_{\mc{h}^+_{\delta, f}}\left( \dot{\mc{h}}^{\pm}_{ \delta, f} \right) =  \pm \left. D H \right|_{\mc{N}^+_{\delta, f}}\left( \log (r) e_z - \mc{h}_{\delta, f}^{\pm} \frac{e_r}{r}\right) , \quad \mr{\mc{tr}}\left(\dot{\mc{h}}^{\pm}_{ \delta, f} \right) = 0
 \]
 and the estimate
 \[
 \sup \left| \dot{\mc{h}}^{\pm}_{ \delta, f} (r, \theta)\right| \leq^{C\epsilon} r^{1 + \gamma}
 \]
  \end{lemma} 
  
  \begin{remark}
  It is not immediately obvious that the terms on the right hand side of the equation satisfied by $\dot{\mc{h}}^{\pm}_{\delta, f}$ are uniformly bounded over the whole disk. It is possible to compute directly that it is, bounded, using the uniform bound $|\mc{h}| \leq^c r^2$.
    \end{remark}
 \begin{proof}
 Fix $\delta > 0$.  Then for $\tau < \delta$,  $\mc{N}^{+}_{\tau} \cap K$ is a normal graph over the domain $\mc{D}' : = \mc{D} \setminus B_{\delta} (0)$. and thus the mapping $\phi_{0}^{\tau, +}: \mc{D}' \rightarrow \mc{N}^+_{\tau}$ is defined. Recalling the notation used in the proof of Proposition \ref{Prop:ImprovedDecay} we have
 \begin{align*}
 \mc{h}' & : = \frac{\mc{h}_{0, \delta, f}^{\tau} \circ \ol{\mc{G}}_{\tau}}{\omega} \\
 &  = \frac{\mc{h}_{\tau, \delta, f} \circ \mc{G}_{\tau}}{\omega}
 \end{align*}
 and 
 \begin{align*}
 \mc{h}^* : = \frac{\mc{h}_{\delta, f} \circ \ol{\mc{G}}_\tau}{\omega}.
 \end{align*}
 Recalling the estimate (\ref{LimitEstimate}) we have
 \[
 \left\|\left( \left. \mc{h}' - \mc{h}^* \right)\right|_{s_0}\right\|^{2, \alpha} \leq^{C\epsilon} \tau
 \]
Thus, as functions on $\mc{D}$ we have the uniform bound 
 \[
 \left|\left. \left(\mc{h}^{\tau, +}_{0, \delta, f} - \mc{h}^+_{\delta, f}\right) \right|_{(r, \theta)}\right| \leq^{C\epsilon} \tau  \]
 as well as local uniform bounds  in $C^{2, \alpha}$. Thus, the difference quotients  $\frac{ \mc{h}^{\tau, +}_{0, \delta, f} - \mc{h}_{\delta, f+}}{\tau}$  and converge locally in $C^{2, \alpha/2}$ as $\tau \rightarrow 0$ to a limiting function $\dot{\mc{h}}^+_{\delta, f}$ satisfying the estimate:
  \[
   \left\|\dot{\mc{h}}^+_{\delta, f} \right\| \leq^{C\epsilon}  r^{1 + \gamma}
 \]
 Moreover since $\mc{N}_{\tau}^+$ is a smooth family with variation field $\dot{\mc{N}}^+$ at $\tau = 0$ given by $- \log(r) e_z$, the variation in the unit normal at $\tau = 0$ is $\dot{N}_{\mc{N}} = \frac{1}{r}e_r$. We then have
 \begin{align*}
 0 & = \frac{H(\mc{N}^+_{\tau, \delta, f}) - H(\mc{D}^+_{\delta, f})}{\tau} \\
 & = \int_0^{1} \left.  D H \right|_{\mc{N}(t)}\left( \frac{\mc{N}^+_{\tau, \delta, f}  - \mc{D}^+_{\delta, f}}{\tau}\right) \\
 & = \int_0^{1} \left.  D H \right|_{\mc{N}(t)}\left(\frac{\mc{N}_{\tau} - \mc{D}}{\tau} + \left(\frac{ \mc{h}^{+}_{\tau, \delta, f} - \mc{h}^+_{\delta, f}}{\tau}\right) e_z + \mc{h}^+_{\tau, \delta, f}\frac{N_{\mc{N}_\tau} - e_z}{\tau}  \right)
 \end{align*}
 where we have set $\mc{N}(t) = (1 -t) \mc{D}_{\delta, f} + t \mc{N}_{\tau, \delta, f}$.
 As $\tau = 0$ we have $\frac{\mc{N}_{\tau} - \mc{D}}{\tau} \rightarrow - \log(r) e_z$, $\left(\frac{ \mc{h}^{+}_{\tau, \delta, f} - \mc{h}^+_{\delta, f}}{\tau}\right) \rightarrow \dot{\mc{h}}^+_{\delta, f}$ and $\frac{N_{\mc{N}_\tau} - e_z}{\tau} \rightarrow \frac{1}{r} e_r$, and the operator  $\int_0^{1} \left.  D H \right|_{\mc{N}}(t) dt$ converges to $D H_{\mc{D}}(-)$. Thus, in the limit as $\tau = 0$, we have
\[
\left. DH_{\mc{D}} \right|_{\mc{h}^+_{\delta, f}}(\dot{\mc{h}}^+_{\delta, f}) + \left. DH_{\mc{D}} \right|_{\mc{h}^+_{\delta, f}}\left( - \log(r) e_z + \mc{h}^+_{\delta, f} \frac{1}{r} e_r\right) = 0, \quad \mr{\mc{tr}} (\dot{\mc{h}}^+_{\delta, f}) = 0.
\]
Finally, we observe that by (\ref{HHigherEsimate}), the higher part of $\dot{\mc{h}}^{+}_{\delta, f}$ satisfies the improved bound 
\[
\left|\ol{\dot{\mc{h}}^{+}_{\delta, f}} \right| \leq^{C\epsilon} r^{1 + \gamma}
\]
The function $\dot{\mc{h}}^{+}_{\delta, f}$ then extends to the disk in $C^2$ and  the higher part satisfies $\mr{\mc{h}} (0) = \nabla \mr{\mc{h}}  (0) = 0$ and thus
\[
\left| \mr{\mc{h}}(r, \theta)\right| \leq^{C\epsilon} r^2.
\]
It then follows that:
\[
\left|\dot{\mc{h}}^{+}_{\delta, f} \right| \leq^{C\epsilon} r^{1 + \gamma}. 
\]
This completes the proof. 
 \end{proof}

\begin{lemma} \label{Prop:TauVariationEquation}
It holds that 
\[
\left. L'_\mc{F}\right|_{\mc{h}'}\left(\frac{\dot{\mc{h}}}{\omega} \right)  = \mc{E}_1+ \mc{E}_2
\]
where
\[
\left\| \mc{E}_2 \right\|^{0, \alpha, \gamma} \leq^C \epsilon
\]
 and where
 \[
  \mc{E}_1 : = \left(\dot{\lambda} + 1 \right) \left. D^2 H' \right|_{\mc{F}}\left(\frac{\xi}{\tau_0}, \mc{h}'_{\tau, \delta, f} \right) 
 \]
 and where above $\xi = s \tanh(s) -1$.
\end{lemma}

\begin{proof}

Set
\[
\mc{G} = \mc{G}_{\tau_0, \delta, f}^{\tau} = \mc{G}_{\tau, \delta, f} \circ \phi_{\tau_0}^{\tau}, \quad \mc{G}_0 : = \mc{G}_{\tau_0, 0, 0}^{\tau}, \quad \mc{h} : = \mc{h}_{\tau_0, \delta, f}^{\tau}, \quad \mc{h}' = \mc{h}/\omega,
\]
and  for $s_0 \in \mb{R}$ and with $\omega_{0} = \omega_{\tau} (s_0)$ set
\[
\left. \tilde{\mc{G}} \right|_{s_0} := \left. \wt{\mc{F}} \right|_{s_0} + \mc{c}_{s_0} \mc{h}' \left. N_{\mc{F}} \right|_{s_0}.
\]
When $\tau  = \tau_0$ we have $\mc{G} = \mc{G}_{\tau_0, \delta, f}$, $\mc{h} = \mc{h}_{\tau_0, \delta, f}$ and $\mc{G}_0 = \mc{G}_{\tau_0} = \tau_0 \mc{F}$. We also have the estimate
\[
\left\|\left. \wt{\mc{G}} \right|_{s_0} - \left. \tilde{\mc{F}} \right|_{s_0} \right\|^{2, \alpha} \leq^C \left\| \left. \mc{h}' \right|_{s_0}\right\| \leq^C \omega_0.
\]
Observe that by construction we have
\[
H \left(\mc{G}\right) =\delta. 
\]
Differentiating gives
\begin{align*}
\left. D H \right|_{\mc{G}}\left( \dot{\mc{G}}\right) = \frac{1}{\tau_0} \left. D H \right|_{\mc{G}/\tau_0}\left( \dot{\mc{G}}/\tau_0\right) =   0,
\end{align*}
where above $\dot{\mc{G}} :=  \left.\frac{d}{d \tau} \mc{G} \right|_{\tau = \tau_0}$.
From (\ref{GtotalVariation}) we have
\[
\dot{\mc{G}} = \left(\dot{\lambda}\tau + 1 \right)\xi N_{\mc{G}} + \mc{h} \dot{N_{\mc{G}}} + \dot{\mc{h}}N_{\mc{G}_0} -   \dot{\lambda} \tau \mc{F}
\]
At $\tau = \tau_0$ we have $\mc{G}/\tau_0 = \mc{F} + \frac{\mc{h}}{\tau_0} N_{\mc{F}}$. This then gives
\begin{align*}
\left. L'_\mc{F}\right|_{\mc{h}'}\left(\frac{\dot{\mc{h}}}{\omega} \right) & = \cosh(s) \left. L_{\mc{F}} \right|_{ \frac{\mc{h}}{\tau_0}} \left( \frac{\dot{\mc{h}}}{\tau_0}\right) \\
& = \cosh(s) \left. DH \right|_{\mc{G}/\tau_0} \left( \frac{\dot{\mc{h}}}{\tau_0} N_\mc{F}\right).
\end{align*}
We thus have
\begin{align*} \label{Eq:Decomposition}
\left. L'_\mc{F}\right|_{\mc{h}'}\left(\frac{\dot{\mc{h}}}{\omega} \right)    & =  - \cosh(s)\dot{\lambda}  \left. D H \right|_{\mc{G}/\tau_0}\left( \mc{F} \right)  -  \left(\tau\dot{\lambda} + 1 \right)\cosh(s)\left. D H \right|_{\mc{G}/\tau_0}\left( \frac{\xi}{\tau_0} N\right) \\
& \quad -  \cosh(s)\left. D H \right|_{\mc{G}/\tau_0}\left(  \frac{\mc{h}}{\tau_0} \dot{N}\right)\\
& = T_1 +T_2 + T_3.
\end{align*}
We will estimate the terms $T_1$, $T_2$, and $T_3$ above separately.  First, observe that 
\[
\left. D H \right|_{\mc{F}} (\xi N_{\mc{F}}) = \left. D H_{\mc{F}} \right|_{0}\left( \xi\right) = 0.
\]
This then gives
\begin{align*}
\cosh(s)\left. D H \right|_{\mc{G}/\tau_0} \left(\frac{\xi}{\tau_0} N \right) 
& =  \cosh(s)  \left. D^2 H \right|_{\mc{F}} \left(\frac{\xi}{\tau_0} N, \frac{\mc{h}}{\tau_0} N \right) + R \\
&  : = \mc{E}_1 + R,
\end{align*}
where $\mc{E}_1$ and $R$ are defined implicitly above. By definition we have
\begin{align*}
\mc{E}_1  &  = \cosh(s)  \left. D^2 H \right|_{\mc{F}} \left(\frac{\xi}{\tau_0} N, \frac{\mc{h}}{\tau_0} N \right) \\ & = D^{2} H'_{\mc{F}} \left(\frac{\xi}{\omega}, \mc{h}' \right).
\end{align*}
With 
\[
\mc{E}_2 = T_1 + T_3 + R,
\]
we then have 
\[
\left. L'_\mc{F}\right|_{\mc{h}'}\left(\frac{\dot{\mc{h}}}{\omega} \right) = \mc{E}_1 + \mc{E}_2.
\]
It remains to estimate $\mc{E}_2$.  We will estimate the terms $T_1$, $T_3$ and $R$ above separately.  Fix  $s_0  \in \mb{R}$.  To estimate $T_1$, observe that since $\mc{F}$ is minimal we have $D H_{\mc{F}} (\mc{F}) = 0$. Then
\begin{align*}
 \cosh(s)\left. \left. D H \right|_{\mc{G}/\tau_0} \left( \mc{F}\right) \right|_{s_0} & =   \mc{c}_{s_0} \left. D H \right|_{\frac{\left. \mc{G} \right|_{s_0}}{\omega_0}} \left( \frac{\left. \mc{F} \right|_{s_0}}{\cosh(s_0)}\right) \\
 & =  \mc{c}_{s_0} \left. D H \right|_{\left. \wt{\mc{G}} \right|_{s_0}} \left( \left. \wt{\mc{F}} \right|_{s_0}\right) \end{align*}
and thus
\[
\left\| \cosh(s)\left. \left. D H \right|_{\mc{G}/\tau_0} \left( \mc{F}\right) \right|_{s_0} \right\|^{0, \alpha} \leq^C \epsilon.
\]
Since $\dot{\lambda} \rightarrow 0$ as $\tau_0 \rightarrow 0$, we have:
\begin{align}\label{T1Estimate}
\left\| T_1 \right\|^{0, \alpha, 1}_{\tau_0} \rightarrow 0
\end{align}
as $\tau_0 \rightarrow 0$.
Next, we estimate the term $T_3$.  Observe that since $\dot{N} = \dot{N}_{\mc{G}_0}$ is tangential to $\mc{F}$ and $\mc{F}$ is minimal  we have 
\[
 \left. D H \right|_{\mc{F}}\left(  \mc{h} \dot{N}\right) = \mc{h}\nabla_{\dot{N}} H_{\mc{F}}= 0.
 \]
with $\mc{c}_0 : =\mc{c}_{s_0}$ this gives:
\begin{align*}
\cosh(s) \left. \left. D H \right|_{\mc{G}/\tau_0}\left(  \frac{\mc{h}}{\tau_0} \dot{N}\right) \right|_{s_0} & =  \mc{c}_{s_0} \left. D H \right|_{\left. \wt{\mc{G}}\right|_{s_0}}\left( \mc{c}_{0} \left. \mc{h}' \right|_{s_0} \left. \dot{N} \right|_{s_0}\right)   \\
& = \mc{c}_{s_0}  \left. D H \right|_{\left. \wt{\mc{G}}\right|_{s_0}}\left( \mc{c}_{0} \left. \mc{h}' \right|_{s_0} \left. \dot{N} \right|_{s_0}\right)    -   \mc{c}_{s_0} \left. D H \right|_{\left. \tilde{\mc{F}} \right|_{s_0}}\left( \mc{c}_{0} \left. \mc{h}' \right|_{s_0} \left. \dot{N} \right|_{s_0}\right) \end{align*}
and thus
\begin{align} \label{Eq:EasierEstiamate}
\left\| \left. T_3 \right|_{s_0}\right\|^{0, \alpha} \leq^C \left\| \left. \mc{h}' \right|_{s_0}\right\|^2 \left\|\left. \dot{N} \right|_{s_0} \right\|  \leq^C \epsilon^2 \omega_0.
\end{align}
Here we have used the expression for $\dot{\phi}_{\tau}$ in (\ref{DotPhiExpression}) to conclude that $\dot{N}_{\mc{G}} = \frac{\partial N_{\mc{F}}}{\partial s} \dot{\phi}_{\tau_0}$ satisfies the bound
\[
\left\|\left.  \dot{N}_{\mc{G}}\right|_{s_0}\right\| \leq^C \frac{1}{\omega_0}.
\]
 To estimate the term $R$, using the integral form of the Taylor remainder and setting $\Xi = \xi N$ we can write explicitly that
\begin{align*}
 R  & :  = \cosh(s) \int^{1}_0 \frac{(1 - t )^2}{2} \left.  D^{3} H \right|_{F_t} \left( \frac{\Xi}{\tau_0},\mc{h} N_{\mc{F}},  \mc{h} N_{\mc{F}}\right)  dt
\end{align*}
where above we have set  $\mc{F}_t : =\mc{F} + t \mc{h} N_{\mc{F}}$. Thus, with  $\wt{\mc{F}}_t : =\left. \wt{\mc{F}} \right|_{s_0} + t \mc{c}_0\mc{h}' N_{\mc{F}}$, we get
\[
\left. R \right|_{s_0} = \mc{c}_0  \int^{1}_0 \frac{(1 - t )^2}{2} \left.  D^{3} H \right|_{\wt{F}_t} \left( \frac{\Xi}{\omega_0},\mc{c}_0 \left. \mc{h}' \right|_{s_0} \left.  N_{\mc{F}} \right|_{s_0},  \mc{c}_0 \left. \mc{h'} \right|_{s_0}  \left. N_{\mc{F}} \right|_{s_0}\right) 
\]

 In order estimate this term, we recall the basic observation made in Section \ref{Sec:GeometricObjects} that the mean curvature operator depends only on first and second derivatives of an immersion or; equivalently, that it is invariant under translations. This is needed since, although the vector field $\Xi$ is large, we have
\begin{align}\label{XiBound}
\left\|\left. \nabla  \Xi \right|_{s_0}\right\|^{1, \alpha} \leq C
\end{align}
for a universal constant $C$. To see this, observe that from the expression for the unit normal $N_\mc{F}$ to $\mc{F}$  in (\ref{CatUnitNormal}) we see
\[
\left\| \left. | \nabla N_{\mc{F}} \right|_{s_0}\right\| \leq^C \cosh^{-1} (s_0).
\]
Since  $\nabla \Xi = \nabla \xi N_{\mc{F}} + \xi \nabla N_{\mc{F}}$ and since $\|\nabla \xi\|^{1, \alpha} \leq C$, we get (\ref{XiBound}).  We thus have
\begin{align} \label{Eq:REstimate}
\left\|\left. R \right|_{s_0} \right\|^{0, \alpha} & \leq^C \left\| \nabla\left( \left. \frac{\Xi}{\omega_0}  \right|_{s_0} \right)\right\|^{1, \alpha} \left(\| \left. \mc{h}' \right|_{s_0}\|^{2, \alpha} \right)^2 \\ \notag
& \leq^C \left(\frac{1}{\omega_0} \right) \left(\epsilon^2\omega^2_0 \right) \\ \notag
& \leq^C \epsilon^2 \omega_0 \notag
\end{align}
Combining this  with the estimates for $T_1$ and $T_3$ in (\ref{T1Estimate}) and (\ref{Eq:EasierEstiamate}) gives. 
\begin{align*}
\left\| \mc{E}_2\right\|^{0, \alpha, 1}_{\tau} \leq^C \epsilon.
\end{align*}

\end{proof}

\begin{lemma}\label{E1Estimates}
It holds that 
\[
\left\| \mc{E}_1\right\|^{0, \alpha, -1}_{1} \leq^C \epsilon, \quad \left\| \ol{\mc{E}}_1\right\|^{0, \alpha, 0}_{1} \leq^{C\epsilon} \tau^\gamma.
\]
\end{lemma}
\begin{proof}
 We will use the notation from the proof of  Lemma \ref{Prop:TauVariationEquation}. We have
\begin{align*}
\left. \mc{E}_1 \right|_{s_0} & =   \left.   D^{2} H'_{\mc{F}} \left(\frac{\xi}{\tau_0}, \mc{h}' \right) \right|_{s_0} \\
& = \mc{c}_0  D^2 H_{\left. \wt{\mc{F}} \right|_{s_0}}\left(\frac{\left. \xi \right|_{s_0}}{\omega_0}, \mc{c}_0 \left. \mc{h}' \right|_{s_0} \right) \\
& =  \mc{c}_{0}  \left. D^2 H\right|_{\left. \wt{\mc{F}} \right|_{s_0} }  \left(\left.\frac{\Xi}{\omega_0} \right|_{s_0}, \mc{c}_{0}\left. \mc{h}' \right|_{s_0} \left. N_{\mc{F}} \right|_{s_0}\right).
\end{align*}
Recall that $\left. \wt{\mc{F}}\right|_{\infty} (x, s) = e^s e_r(x)$ and thus:
\[
 \left. D^2 H_{\left. \wt{\mc{F}} \right|_{\infty}}\right|_{0} = 0.
 \]
 Using the estimate
 \begin{align} \label{FDifEst}
 \left\| \left. \wt{\mc{F}} \right|_{s_0} -\left.  \wt{\mc{F}} \right|_{\infty} \right\| \leq^C \cosh^{-1} (s_0) ,
\end{align}
 the bound (\ref{XiBound}) and the smoothness of the mean curvature on $R^{k, \alpha} (\Omega_0)$ gives
\begin{align*}
\left\| \left. \mc{E}_1 \right|_{s_0} \right\|^{0, \alpha} & \leq^C \left\| \left. D^2 H \right|_{\left. \wt{\mc{F}} \right|_{s_0}}  \right\|\left\| \nabla \left(\frac{\left. \Xi \right|_{s_0}}{\omega_0}\right)\right\|^{1, \alpha} \left\|\left.  \mc{h}' \right|_{s_0} \right\|^{2, \alpha} \\
& \leq^C \left( \cosh^{-1} (s_0)\right)\left(\frac{1}{\omega_0} \right) \left( \epsilon \omega_0\right) \\
& \leq^C \epsilon \cosh^{-1} (s_0).
\end{align*}
This completes the estimate for $\mc{E}_1$. To estimate $\ol{\mc{E}_1}$, write:
\begin{align*}
\left. D^2 H_{\left. \wt{\mc{F}} \right|_{s_0}} \right|_{0 }  \left(\frac{ \left. \xi \right|_{s_0}}{\omega_0}, \mc{c}_{s_0}
\left. \mc{h}' \right|_{s_0} \right) & = \left. D^2 H_{\left. \wt{\mc{F}} \right|_{s_0}} \right|_{0 }  \left(\frac{ \left. \xi \right|_{s_0}}{\omega_0}, \mc{c}_{s_0}
\left. \ol{\mc{h}'} \right|_{s_0} \right) + \left. D^2 H_{\left. \wt{\mc{F}} \right|_{s_0}} \right|_{0 }  \left(\frac{ \left. \xi \right|_{s_0}}{\omega_0}, \mc{c}_{s_0}
\left. \mr{ \mc{h}}' \right|_{s_0} \right) \\
& = T_1 + T_2
\end{align*}
where above $ \ol{\mc{h}'} $ and $ \mr{\mc{h}'}$ denote the higher and lower parts of $\mc{h}'$ respectively. From the estimate for $ \ol{\mc{h}'} $ in Proposition \ref{Prop:ImprovedDecay}, we have that 
\begin{align*}
\left\| T_1  \right\| & \leq^C \left\|\left. D^2 H \right|_{\left. \wt{F} \right|_{s_0}} \right\| \left(\frac{1}{\omega_0} \right)\left( \tau\omega_0^\gamma\right)\\
& \leq^C \tau^{\gamma} 
\end{align*}
where above we have used  (\ref{FDifEst}) to bound $\left\|\left. D^2 H \right|_{\left. \wt{F} \right|_{s_0}} \right\|$.
We now claim
\begin{align}\label{T2Orthogonality}
\ol{T}_2 = 0.
\end{align}
This follows from the fact that, since $\mc{F}$ is rotationally symmetric,  the operator $\left. D^2 H_{\left. \wt{\mc{F}} \right|_{s_0}} \right|_{0 }  \left(\frac{ \left. \xi \right|_{s_0}}{\omega_0}, - \right)$ is mode preserving (recall Section \ref{Sec:ModePreservation}) and the fact that $\mr{\mc{h}}'$ is orthogonal to the lower modes. This completes the proof. 
\end{proof}

\begin{lemma} \label{Eq:HDotTotalBound}
There are functions $\dot{\mc{h}}_1$ and $\dot{\mc{h}}_2$ such that 
\[
\left\| \dot{\mc{h}}'_1\right\|^{2, \alpha, -\gamma}_{1} \leq^C \epsilon, \quad \left\| \dot{\mc{h}}'_2\right\|^{2, \alpha, \gamma}_{\tau} \leq^C \epsilon 
\]
and such that $ \dot{\mc{h}} = \omega\left(\dot{\mc{h}}'_1  + \dot{\mc{h}}'_2 \right)$.
\end{lemma}

\begin{proof}
 By Proposition \ref{Prop:WeightedInvertibilty}, there is a unique function $\dot{\mc{h}}'_2 \in W^{2,\alpha, \gamma}_{\tau, n}$ such that 
 \[
 L_F' (\dot{\mc{h}}'_2) = \mc{E}_2, \quad \mr{\mc{tr}} (\dot{\mc{h}}'_2) = 0,
 \]
 and satisfying the estimate
 \[
 \| \dot{\mc{h}}'_2\|^{2, \alpha, \gamma}_{\tau} \leq^C \| \mc{E}_2\|^{0, \alpha, \gamma}_{\tau} \leq^C \epsilon,
 \]
 where $C > 0$ is a universal constant. Next, we solve for the function $\mc{E}_1$. Write:
 \[
 \mc{E}_1 : = \mr{\mc{E}}_1 + \ol{\mc{E}}_1,
 \]
 where $ \ol{\mc{E}}_1$ and $ \mr{\mc{E}}_1$ denote the lower and higher parts of $\mc{E}_1$, respectively. By  Proposition \ref{Prop:DecayInvertibility}, the estimate for $\mc{E}_1$ in Lemma \ref{E1Estimates} and since $\left\| \mr{\mc{E}}_1\right\| \leq \left\|\mc{E}_1\right\|$, there is a unique function $\mr{\mc{h}}'_1$ such that 
\[
L'_{\mc{F}} \mr{\mc{h}}'_1 = \mr{\mc{E}}_1, \quad \mc{tr} (\mr{\mc{h}}') = 0
\]
and satisfying the estimate
\[
\left\| \mr{\mc{h}}'_{1}\right\|^{2, \alpha, -\gamma}_{1} \leq^C \epsilon.
\]
Solving for $\ol{\mc{E}}_1$ can be done by separating variables and integrating as follows: We have
 \[
 \ol{\mc{E}}_1 =  a(s) + b(s) \cos(x) + c(s) \sin (x).
 \] 
  for class $C^{0, \alpha}$ funtions $a$, $b$ and $c$. Moreover, by the estimate for $\ol{\mc{E}_1}$ in Lemma (\ref{E1Estimates}),  the coefficients $a$, $b$ and $c$ satisfy the bounds
  \[
  |a|, |b|, |c| \leq^C\epsilon \tau^\gamma.
  \]
   Since $\frac{\tanh(s)}{\cosh(s)}$ is in the kernel of $L_{\mc{F}}'$, variation of parameters gives 
  \[
  \ol{\mc{h}}'_{a} (s) : = \frac{\tanh{s}}{\cosh(s)} \int_0^{s} \tanh^{-2}(s') \int_0^{s'} \sinh(s'') a(s'') ds'' ds'
  \]
  as a solution to the equation 
  \[
  L'_\mc{F} \ol{\mc{h}}'_{a}= a.
  \]
  Moreover, we clearly have 
  \[
  \ol{\mc{h}}'_{a}(0) = \left. \frac{\partial \ol{\mc{h}}'_{a}}{\partial s} \right|_{(x, 0)}= 0,
  \]
   and the estimate:
  \[
  \left| \ol{\mc{h}}'_{a}(s)\right| \leq^C \| a \|/\cosh(s) \leq^C \tau^\gamma \epsilon/\cosh(s).
  \]
Solving for the term $b(s) \cos(x)$ is similar,  and variation of parameters yields the solution
\[
\ol{\mc{h}}'_{b} = \left(\cosh^{-2} (s) \int_{0}^{s} \cosh^2(s') \int_{0}^{s'} b(s'')  ds'' ds'\right) \cos(x), 
\]
which is again normalized and satisfies the estimate 
\[
\left\| \ol{\mc{h}}'_{b}\right\|^{2, \alpha} \leq^C \| b\| \leq^C \tau^\gamma \epsilon.
\]
A function $\ol{\mc{h}}'_{c}$ solving $L_\mc{F}' \ol{\mc{h}}'_{c}= c \sin(x)$ and given by the expression above--with $\sin(x)$ substituted for $\cos(x)$--can be found with the same methods.  The function $\ol{\mc{h}}'_1 := \ol{\mc{h}}'_{a} + \ol{\mc{h}}'_{b} + \ol{\mc{h}}'_{c}$ then satisfies
\begin{align}
L'_\mc{F} \ol{\mc{h}}'_{1} = \ol{\mc{E}}_1.
\end{align}
and the estimate 
\begin{align}
\left\|\ol{\mc{h}}'_{1}\right\|^{2, \alpha} \leq^C \tau^\gamma \epsilon
\end{align}
and thus
\begin{align*}
\left\|\ol{\mc{h}}'_{1}\right\|^{2, \alpha, -\gamma}_{1} \leq^C \epsilon.
\end{align*}
The function $\mc{h}'_1 : = \ol{\mc{h}}'_1 + \mr{\mc{h}}'_1$ then satisfies 
\[
\mc{L}'_{\mc{F}}\mc{h}'_1 = \mc{E}_1
\]
and 
\[
\left\| \mr{\mc{h}}'_{1}\right\|^{2, \alpha, -\gamma}_{1} \leq^C \epsilon.
\]
We then have
 \[
 L'_\mc{F}\left( \frac{\dot{\mc{h}}}{\omega}  +  \mc{h}'_1  + \mc{h}'_2 \right) = 0, \quad \mr{\mc{tr}} \left( \frac{\dot{\mc{h}}}{\omega}  +  \mc{h}'_1  + \mc{h}'_2 \right) = 0
 \]
 Since $\dot{\mc{h}}$,  $\mc{h}'_1$ and $\dot{\mc{h}}'_2$ are normalized, by Proposition \ref{Prop:WeightedInvertibilty} we have 
 \[
 \dot{\mc{h}} =   - \omega \left(\mc{h}'_1 + \mc{h}'_2\right)
 \]
 This completes the proof. 
  \end{proof}

\begin{proof}[Proof of Proposition \ref{Prop:HDotBounds}/Theorem \ref{Prop:MainTheoremDifferentiability}]
We will identify the functions $\dot{\mc{h}} = \dot{\mc{h}}_{\tau_0}$ with their push-forwards to $\mc{N}_{\tau_0}$ under $\mc{G}_{\tau_0}$. Then the bound in Lemma \ref{Eq:HDotTotalBound} translates to 
 \begin{align}\label{Eq:HDotTotalBoundAgain}
 \left| \dot{\mc{h}} (r, \theta)\right| \leq^C \frac{\tau^\gamma}{r^\gamma} + r^{1 + \gamma}
 \end{align}
 which holds on compact subsets of $\mc{N}_{\tau_0}$ away from the origin.  Consider a sequence $\tau_j \rightarrow 0$. Then the functions $\dot{\mc{h}}^+_j : = \dot{\mc{h}}^+_{\tau_j}$ satisfy 
   \[
   \left. DH_{\mc{N}_{\tau}}\right|_{\mc{h}_j^+}\left( \dot{\mc{h}}^{+}_{j}\right) =-  \left. D H \right|_{\mc{N}^+_{j}}\left(\dot{\mc{N}}^+_{j} + \mc{h}_j \dot{N}_{\mc{N}} \right)
   \]
 The variation field generated by the family $\mc{N}^+_{\tau}$ at $\tau = 0$ is $-  \log(r) e_z$, and thus 
  \[
  \dot{N}_{\mc{N}^+_{j}} \rightarrow \frac{1}{r} e_r
  \]
  as $j \rightarrow 0$ smoothly away from the origin.  The functions $\dot{\mc{h}}^{+}_j$ then converge  in $C^{2, \alpha/2}$ to a limiting function $\dot{\mc{h}}^+_{0}$ on the punctured  unit disk $\mc{D}$ satisfying the limiting equation
 \begin{align}
 \left. D H_\mc{D} \right|_{\mc{h}}\left(  \dot{\mc{h}}^\pm_0 \right) = \mc{E}, \quad  \mr{\mc{tr}}(\dot{\mc{h}}^\pm) = 0,
  \end{align}
  where 
  \[
  \mc{E} = \pm \left. D H \right|_{\mc{N}^+_{\delta, f}}\left( +\log (r) e_z - \mc{h}^+ \frac{e_r}{r}\right)
  \]
  and the estimate 
  \[
   \left\|\dot{\mc{h}}^+ \right\| \leq^C r^{1 + \gamma}.
  \]
Thus, $\dot{\mc{h}}^+$ is normalized and solves the same modified Poission-Dirichlet problem on $\mc{D}$ as $\dot{\mc{h}}^+_{\delta, f}$. Since solutions are unique (Lemma \ref{ModifiedDiskProblem}), it follows that the family of functions  $\dot{\mc{h}}^\pm_{\tau, \delta, f}$ extends continuously to $\tau = 0$.
  \end{proof}
  \subsection{Criteria for embeddedness}
 When $\delta = 0$, the surfaces $\mc{N}_{\tau, 0, f}$ are  minimal surfaces, and thus by the maximum principle they are embedded provided the boundary is.  If we orient the surface $\mc{N}_{\tau, \delta, f}$ by the downward pointing unit normal, then for $\delta > 0$ the surfaces $\mc{N}_{\tau, \delta, 0}$, which are rotationally symmetric about the $z$-axis and under reflexions through the plane $z = 0$, are embedded. Moreover, as a direct consequence of  the estimate (\ref{LimitEstimate}), it follows that  the surface $\mc{N}_{\tau, \delta, f}$ is embedded provided $f$ is small in terms of $\delta$.

\bibliographystyle{amsalpha}

\end{document}